\documentclass{article}
\usepackage{graphicx} 
\usepackage[a4paper, textwidth=16cm]{geometry}
\usepackage{booktabs}
\usepackage{float} 
\usepackage{mathrsfs}
\usepackage{amsthm,amssymb}
\newtheorem{theorem}{Theorem}
\newtheorem{lemma}{Lemma}
\newtheorem{proposition}{Proposition}
\newtheorem{remark}{Remark}
\newtheorem{corollary}{Corollary}
\newtheorem{definition}{Definition}
\usepackage{amsmath}
\usepackage{xcolor}

\usepackage{booktabs,tabularx,array}
\usepackage{hyperref}
\usepackage{microtype}

\hypersetup{
  colorlinks=true,
  linkcolor=blue,
  citecolor=blue,
  urlcolor=blue
}

\newcommand{\PG}{\mathrm{PG}}

\newcommand{\F}{\mathbb{F}}

\newcommand{\Gscr}{\mathscr{G}}
\newcommand{\E}{\mathbb{E}}
\newcommand{\G}{\mathcal{G}}

\newcommand{\Stab}{\operatorname{Stab}}

\title{Weak arcs and applications to the DNA-based storage access problem}
\author{Geertrui Van de Voorde 
\thanks{University of Canterbury, Christchurch, New Zealand. Email: geertrui.vandevoorde@canterbury.ac.nz.}
\footnotemark[3]
\and Ferdinando Zullo 
\thanks{Universit\`a degli studi della Campania Luigi Vanvitelli, Caserta, Italy.
Email: ferdinando.zullo@unicampania.it.}
\thanks{Supported by the Marsden Fund Council, managed by Royal Society Te Ap\={a}rangi (MFP-UOC2515)} 
\thanks{
Partially supported by the Italian National Group for Algebraic and Geometric Structures and their Applications (GNSAGA - INdAM)}}
\date{}

\begin{document}

\maketitle
\begin{abstract}
    {\em Weak arcs} are point sets in $\PG(n-1,q)$ meeting every {\em general} hyperplane (those are the hyperplanes not going through one of the points given by the standard basis vectors) in at most $n-1$ points. 
    
    In this paper, we study weak arcs together with balanced variants which are contained on the sides of the fundamental simplex. We give an upper bound on the size of weak arcs, characterise the largest balanced quasi-arcs in the plane and construct large balanced quasi-arcs in $\PG(3, q)$. We then use these configurations to build point sets for the random-access problem in DNA-based storage. The constructions are explicit, work over small fields, and attain recovery expectations matching the best known asymptotic bounds.
\end{abstract}
\section{Introduction}
\subsection{A geometric view on the DNA access problem}

DNA-based data storage is a promising medium for long-term archival storage. In a typical system, digital data are encoded into DNA strands, synthesised, stored as a pool of molecules, and later recovered by sampling and sequencing strands from that pool. This gives rise to coding and recovery problems, see \cite{YazdiEtAl2015,OrganickEtAl2018,BarLevSabaryGabrysYaakobi2025} and their references.

Our paper is motivated by one of these questions, usually called the \emph{random-access problem}. One wishes to recover a specified information strand after reading as few encoded strands as possible. We use a linear algebraic model introduced in \cite{GruicaBarLevRavagnaniYaakobi2025} and developed further in \cite{GruicaMontanucciZullo2026}.

Here, the original information consists of $n$ strands over a finite field $\mathbb F_q$. Each encoded strand contains a linear combination of these information strands and is represented by its coefficient vector in $\mathbb F_q^n$. The storage system is represented by a multiset $\mathcal G$ of nonzero vectors, where repetition records the multiplicity with which an encoded strand occurs. 

Let $e_1,\ldots,e_n$ be the standard basis vectors of $\mathbb F_q^n$. The $i$-th information strand can be recovered from a collection of sampled strands precisely when their coefficient vectors span $e_i$. If the elements of $\mathcal G$ are read in a uniformly random order, with replacement, let $\tau_{e_i}(\mathcal G)$ be the first time at which the vectors read so far span $e_i$. The goal is to choose the support and multiplicities of $\mathcal G$ so as to minimise
\[
    M(\mathcal G)
    :=\max_{1\leq i\leq n}\mathbb E[\tau_{e_i}(\mathcal G)],
\]
the worst-case expected number of reads needed to recover any single information strand.

It is natural to regard the elements of $\mathcal G$ as points of $\mathrm{PG}(n-1,q)$. The standard basis vectors correspond to the fundamental points $P_1,P_2,\ldots,P_n$; we denote the corresponding value by $\tau_{P_i}(\mathcal G)$.

The DNA-storage problem thus becomes: find multisets of points in $\mathrm{PG}(n-1,q)$ for which
\[
    \max_{1\leq i\leq n}\mathbb E[\tau_{P_i}(\mathcal G)]
\]
is small.
Several constructions for good point sets, that is, those with \[\max_{1\leq i\leq n}\mathbb E[\tau_{P_i}(\mathcal G)]<n,\]  have appeared in the literature in the previous years. The best overall bound at the time of writing is given in \cite{WangYaakobi2026} (see also \cite{BodurEtAl2025}). Their construction for $n=4$ has {\em non-integer} multiplicities for strands, and an optimal value of \[
 \limsup_{q\to\infty}M(\Gscr(q,4))<0.862882\cdot4.
\]
We will use this result as a benchmark for our constructions (see Section \ref{sec:constructions}).

\subsection{Weak arcs}

The previous problem suggests distributing the points so that subspaces avoiding a fundamental point contain relatively few points of $\mathcal G$, and motivates the following definition.

\begin{definition}
Let $P_1,\ldots,P_n$ be points in general position in $\mathrm{PG}(n-1,q)$. A point set $S$ is a \emph{weak arc with respect to $P_1,\ldots,P_n$} if every hyperplane $H$ not containing a fundamental point satisfies
\[
    |H\cap S|\leq n-1.
\]
\end{definition}

Every arc, that is, every point set in general position, is a weak arc, but the converse need not hold. When the points $P_1,\ldots,P_n$ are represented by the standard basis vectors, the hyperplanes not through any of them have equations
\[
    a_1X_1+\cdots+a_nX_n=0
\]
with every coefficient $a_i$ nonzero. We call these hyperplanes \emph{general hyperplanes}.

Of particular interest are point sets that are contained on the edges of the fundamental simplex (the edges are the lines of the form $P_iP_j$). We call such a configuration a \emph{balanced quasi-arc} of parameter $a$ when it contains the fundamental points and exactly $a$ non-fundamental points on every edge. This is a generalisation of the planar balanced quasi-arcs as introduced in \cite {GruicaMontanucciZullo2026} in which case they are contained in the sides of the fundamental triangle. 

\subsection{Overview of this paper}
In the first part of this paper, we will study weak arcs in $\PG(n-1,q)$ from a geometric point of view, before considering the application to the DNA problem in the second part of this paper. 

More precisely, in Subsection \ref{subs:size}, we derive an upper bound on the size of weak arcs in $\PG(n,q)$, as well as a probabilistic construction of a quasi-arc contained on the edges of the fundamental simplex.
We then focus on the planar case in Subsection \ref{subs:planar}, where we  provide a characterisation of the largest possible balanced quasi-arcs. For $q$ odd, these arise from the squares in $\F_q$, but for $q$ even, the situation is different.

In Subsection \ref{subs:higher} we generalise the planar balanced quasi-arc constructed in Subsection \ref{subs:planar} to a particularly nice {\em weak  cap}, which is a point set such that every line not through one of the fundamental points contains at most two points of the set. We then show that a well-chosen subset of this weak cap is a {\em weak arc}, which is contained on the sides of the fundamental tetrahedron.

In Section \ref{sec:arcDNA} we link the weak arcs to the random access problem. We first give an alternative expression for the value $\E[\tau_P]$ in Subsection \ref{subs:alt}, and then, in Subsection \ref{subs:opt} give evidence to the claim that quasi-arcs are a good choice as point set for the DNA problem: we show that balanced quasi-arcs of parameter $(q-1)/2$ are the optimal choice of points when taking a balanced point set contained in the edges of the fundamental triangle.

In Section \ref{sec:constructions}, we will consider different point sets, with multiplicities and show that they have similarly low expectation or sometimes lower expectation to the examples found in earlier work. Our geometric constructions have the advantage that they work for any $q$ and can be made as large/small as desired, to a certain extent, independently of $q$. More details of those comparisons are provided in Section \ref{subs:WY}.

\section{Weak arcs}

\subsection{On the size of weak arcs}\label{subs:size}
\subsubsection{An upper bound on the size of a weak arc}

Let $P_0,P_1,\dots,P_n$ denote the fundamental points of
$\PG(n,q)$. A hyperplane
$\pi\subset\PG(n,q)$ with equation $a_0x_0+\cdots+a_nx_n=0$ is called
\emph{general} if it does not contain any of the $P_i$'s. This happens if and only if all coefficients $a_i$ are nonzero. It is clear that the number of general hyperplanes in $\PG(n,q)$ equals $(q-1)^n$. We first count how many general hyperplanes there are through a point in $\PG(n,q)$, which will depend on the position of that point with respect to the fundamental points. For a point $Q=(x_0,\ldots,x_n)\in\PG(n,q)$, this is determined by its \emph{support} $J(Q)$ where
\[
  J(Q)=\{\,i\in\{0,\dots,n\}\,:\,x_i\neq 0\,\},\qquad j(Q)=|J(Q)|\in\{1,\dots,n+1\}.
\]
We see that $k=1$ corresponds to fundamental points, $k=2$ to points on
a line through two fundamental points but not a fundamental point itself, and so on, with $k=n+1$ for points not contained in any coordinate hyperplane, that is, not spanned by any set of $n$ fundamental points.

For a point set $S\subseteq\PG(n,q)$ we write $N_k=|\{Q\in S:j(Q)=k\}|$, so
$|S|=\sum_{k=1}^{n+1}N_k$.

\begin{lemma} \label{lem:sumset}
Let $q$ be a prime power and let $b_0,b_1,\dots,b_s \in \mathrm{GF}(q)^*$ be
fixed nonzero elements. The number of tuples
$(a_0,a_1,\dots,a_s) \in (\mathrm{GF}(q)^*)^{s+1}$ satisfying
\[
    \sum_{i=0}^{s} a_i b_i \;=\; 0
\]
is independent of the choice of $b_0,b_1,\dots,b_s$ and equals
\[
    N_{s+1}(q) \;=\; \frac{(q-1)^{s+1} + (-1)^{s+1}(q-1)}{q}.
\]
\end{lemma}

\begin{proof}
First note that since $b_0,\ldots,b_s$ are fixed non-zero elements, we may consider the bijection $c_i=a_i\cdot b_i$ for all $i$, and equivalently count the number $N_{s+1}(q)$ of non-zero elements $c_0,\ldots,c_s$ such that $\sum_{i=0}^{s} c_i = 0.$

We can find $N_{k+1}(q)$ for general $k \geq 0$ by deriving a recurrence relation. Clearly $N_1(q)=0$. For $k \geq 1$, choose $c_0,c_1,\dots,c_{k-1} \in \mathrm{GF}(q)^*$ arbitrarily; this
gives $(q-1)^{k}$ possibilities. The final coordinate is then forced to be
$c_k = -(c_0 + \cdots + c_{k-1})$, and this is an admissible choice precisely when
$c_0 + \cdots + c_{k-1} \neq 0$, so the number of bad tuples is exactly $N_{k}(q)$. Hence we find the recurrence relation
\begin{align*}
    N_{k+1}(q) \;=\; (q-1)^{k} - N_{k}(q),
\end{align*}

which can be solved, using that $N_1(q)=0$, to find
\[
    N_{k+1}(q) \;=\; \frac{(q-1)^{k+1} + (-1)^{k+1} (q-1)}{q}. \qedhere
\]
\end{proof}
\begin{lemma}\label{lem:fj}
For $Q\in\PG(n,q)$ with $j(Q)=j$, the number of general hyperplanes through
$Q$ is
\[
  f_j \;=\; \frac{(q-1)^n+(-1)^j(q-1)^{n+1-j}}{q}.
\]
In particular $f_1=0$, $f_2=(q-1)^{n-1}$, and
$f_3=(q-1)^{n-2}(q-2)$.
\end{lemma}

\begin{proof}
Consider a point $Q=(x_0,\ldots,x_n)$ with $J=J(Q)$, $|J|=j$. A general hyperplane with equation $a_0x_0+\cdots a_nx_n=0$ contains $Q$ if and only if 
 $\sum_{i\in J}a_ix_i=0$. The coefficients $a_i$ not contained in $J$ are free, so there are $(q-1)^{n+1-j}$ choices for them. The number of choices for $a_i$ such that $\sum_{i\in J}a_ix_i=0$ is given in Lemma \ref{lem:sumset}, and equals $N_{j}(q)= \frac{(q-1)^{j} + (-1)^{j} (q-1)}{q}.$
 This shows that the number of choices for $(a_0,\ldots,a_n)$ determining a general hyperplane through $Q$ is given by $(q-1)^{n+1-j}(\frac{(q-1)^{j} + (-1)^{j} (q-1)}{q})$. Since every general hyperplane is determined by $(q-1)$ such tuples, we find that the number of general hyperplanes is given by $(q-1)^{n-j}(\frac{(q-1)^{j} + (-1)^{j} (q-1)}{q}).$
\end{proof}

It now directly follows that the non-fundamental points that are lying on the smallest number of general hyperplanes are those lying on a plane spanned by $3$ fundamental points (but not on an edge).
\begin{corollary}\label{lem:fmin}
For every $j\in\{2,3,\dots,n+1\}$ we have $f_j\ge f_3=(q-1)^{n-2}(q-2)$.
\end{corollary}

\begin{theorem}\label{thm:upperboundweakarc}
Let $S$ be a weak arc in $\PG(n,q)$, $n\ge 2$ and $q\ge n+3$. Then
\[
  |S|\;\le\;nq+(n+1).
\]
\end{theorem}

\begin{proof}
We double count couples $(Q,\pi)$ with $Q\in S$, $\pi$ a general hyperplane,
$Q\in\pi$:
\[
  \sum_{k=1}^{n+1}N_k f_k
  \;=\;\sum_{\pi\text{ gen.}}|S\cap\pi|
  \;\le\;n(q-1)^n.
\]
Where we have used that every general hyperplane has at most $n$ points in $S$.
Since $f_1=0$ this reads $\sum_{k=2}^{n+1}N_k f_k\le n(q-1)^n$. Recall that $|S|=\sum_{k=1}^{n+1} N_k$.
By
Lemma~\ref{lem:fmin}, $f_k\ge f_3=(q-1)^{n-2}(q-2)$ for $k\ge 2$, so
\[
  (|S|-N_1)\,(q-1)^{n-2}(q-2)\;\le\;n(q-1)^n,
\]
hence
\[
  |S|-N_1\;\le\;\frac{n(q-1)^2}{q-2}
\]

Since $N_1\le n+1$,
\[
  |S|\;\le\;nq+(n+1)+\frac{n}{q-2}.
\]
For $q\ge n+3$, we find that
$|S|\le nq+(n+1)$.
\end{proof}

The bound can be attained as follows: consider $n$ lines $\ell_1,\dots,\ell_n$ through a fundamental point (say $P_0$) where the lines do not contain any further fundamental point. Put
\[
  S=\ell_1\cup\cdots\cup\ell_n\cup\{P_1,\dots,P_n\}.
\]
Then $|S|=nq+(n+1)$, and every general hyperplane $\pi$ meets $S$ in exactly $n$ points.
When $q$ is large enough, at least for $n=2$, it can be shown that the converse is true \cite{Cassie}: a weak arc in $\PG(2,q)$ of size $2q+3$ consists of two lines through exactly one fundamental point, together with the remaining fundamental points.

\subsubsection{A probabilistic construction of a balanced quasi-arc}

In recent years, it has become common practice to construct certain subsets in projective planes or spaces (small complete arcs, strong blocking sets, etc...) using the probabilistic method. 
Before turning to explicit constructions, it is therefore natural to ask what happens if the points on the edges of the fundamental simplex are chosen at random. 

As the following proposition shows, this approach does produce balanced quasi-arcs, but only of relatively small size. In fact, the maximum allowed parameter grows much more slowly in $q$ than in the explicit constructions that we will obtain in the next Sections; the goal of this subsection is merely to derive a general base-line result for the existence of balanced quasi-arcs in $\PG(n,q)$.

\begin{remark}
There is little point applying the probabilistic method to construct weak arcs without any further restrictions. We have seen that it is easy to take a union of $n$ lines through a fundamental point, together with the fundamental points, to obtain a weak arc of size $nq+n+1$. The probabilistic method in $\PG(n,q)$ will construct a weak arc of size order $q^{1/n}$, far below this value. This is roughly the same order as found when constructing an arc with this method; the fact that we only constrain the hyperplanes not through the fundamental points does not structurally improve the size of the obtained point set. 
\end{remark}




In view of this remark, we now check what happens if we restrict ourselves to the sides of the fundamental simplex.
Let $P_1,P_2,\ldots,P_{n+1}$ be the fundamental points in $\PG(n,q)$, and say that a {\em balanced quasi-arc} of parameter $a$ in $\PG(n,q)$ is a weak arc whose points all lie on the sides of the fundamental simplex.
Just like for weak arcs, we will find a point set of order $q^{1/n}$ but in this case, there is no trivial construction that immediately does better. In the next sections though, we will construct much larger balanced quasi-arcs in $\PG(2,q)$ and $\PG(3,q).$

\begin{lemma}\label{lem:hyperplanes}
Let $m=\binom{n+1}{2}$, the number of edges of the fundamental simplex and let $V$ be the set of points on an edge of the simplex, different from the fundamental points. Then the number $B$ of
$(n+1)$-subsets of $V$ that lie on a general hyperplane is at most
\[
B\;\le\;(q-1)^n\binom{m}{n+1}.
\]
\end{lemma}

\begin{proof}
There are $(q-1)^n$ general hyperplanes. Each such
hyperplane meets each of the $m$ edges in a point  of $V$. Hence, it contains exactly $\binom{m}{n+1}$  $(n+1)$-subsets in $V$. 
\end{proof}

\begin{proposition}\label{thm:balanced}
Let $q$ be a prime power, $n\ge2$, $m=\binom{n+1}{2}$, and put
\[
c_n\;=\;\tfrac12\bigl(2\tbinom{m}{n+1}\bigr)^{-1/n},
\qquad
k\;=\;\bigl\lfloor 2c_n\,(q-1)^{1/n}\bigr\rfloor .
\]
Then $\PG(n,q)$ contains a balanced quasi-arc of parameter $a$ for
\emph{every} integer $a$ with
\[
1\;\le\;a\;\le\;\Bigl\lceil \tfrac{k}{2}\Bigr\rceil .
\]
\end{proposition}

\begin{proof}
Note first that $2c_n=(2\binom{m}{n+1})^{-1/n}$, so
$k=\lfloor(2\binom{m}{n+1})^{-1/n}(q-1)^{1/n}\rfloor$ and hence
$k^{\,n}\le \frac{q-1}{2\binom{m}{n+1}}$.

On each line $P_iP_j$, choose independently, a uniformly random $k$-subset
$A_{ij}$ of points different from $P_i,P_j$, and set $R=\bigcup_{i<j}A_{ij}$; by
construction $R$ has exactly $k$ points on every edge.

Now let $B_R$ be the number of bad $(n+1)$-subsets contained in $R$: those are the $(n+1)$-subsets where each point belongs to a different edge.
On each edge $P_iP_j$, a point lies in
$A_{ij}$ with probability $k/(q-1)$, and this is independent across edges. Hence,
using linearity of expectation, and that $B\le(q-1)^n\binom{m}{n+1}$ (Lemma \ref{lem:hyperplanes}),
\[
\E[B_R]\;\le\;B\Bigl(\frac{k}{q-1}\Bigr)^{\!n+1}
\;\le\;\binom{m}{n+1}\frac{k^{\,n+1}}{q-1}\;=:\;\mu(k).
\]
Since $k^{\,n}\le \frac{q-1}{2\binom{m}{n+1}}$,
\[
\mu(k)\;=\;\binom{m}{n+1}\frac{k^{\,n}}{q-1}\cdot k\;\le\;\frac{k}{2}.
\]
This shows that there is a choice of $R$ such that $B_R\le\lfloor\mu(k)\rfloor\le k/2$. Fix this $R$ and delete one point from each bad $(n+1)-$ set. Since we need to delete at most $\lfloor\mu(k)\rfloor$ points in total, we certainly delete at most $\lfloor\mu(k)\rfloor$ on any single edge. Each edge still has at least
\[
k-\lfloor\mu(k)\rfloor\;\ge\;k-\tfrac{k}{2}\;=\;\tfrac{k}{2}
\]
points. We can remove points on each edge to exactly $\lceil k/2\rceil$ points to create a balanced quasi-arc of parameter
$\lceil k/2\rceil$.
\end{proof}
\begin{remark}
    A weaker but easier bound on $k$ can be produced from the expressions in Proposition \ref{thm:balanced}, to show that there exists a balanced quasi-arc with parameter $a$ for $a$ at most $\frac{(q-1)^n}{9n}$.
\end{remark}

\subsection{Balanced quasi-arcs in the plane}\label{subs:planar}

In this section, we will characterise large (balanced) quasi-arcs in the plane. Recall that these are weak arcs whose point set lies on the sides of a triangle. 
We first provide a trivial upper bound on the number of points $x$ on each side of a balanced quasi-arc, see \cite[Proposition 9]{GruicaMontanucciZullo2026} for a proof.

\begin{lemma}\label{prop:univ}
For every prime power $q$, any balanced quasi-arc of parameter $a$ satisfies
\[
  a\ \le\ \Bigl\lfloor \tfrac{q-1}{2}\Bigr\rfloor .
\]
\end{lemma}

In order to characterise the largest balanced quasi-arcs, we parametrise the non-vertex points on the triangle by
\[
  P_1(a)=(1,-a,0),\qquad P_2(b)=(0,1,b),\qquad P_3(c)=(1,0,c),
  \qquad a,b,c\in \F_q^* .
\]
Let $\F_q^\ast$ represent the multiplicative group of the field $\F_q$.  The following easy Lemma translates the geometric property of being a weak arc into a property of subsets of $\F_q^*$.

\begin{lemma}\label{prop:reduction}
Choose $A,B,C\subseteq \F_q*$ and let
\[
  S=\{P_1(a):a\in A\}\cup\{P_2(b):b\in B\}\cup\{P_3(c):c\in C\}.
\]
Then $S$ is a weak arc in $\PG(2,q)$ if and only if
\[
  A\cdot B\;\cap\;C=\varnothing,
\]
where $A\cdot B=\{ab:a\in A,\ b\in B\}$. In particular a balanced quasi-arc of
parameter $x$ exists if and only if there are sets $A,B,C\subseteq \F_q^\ast$ with
$|A|=|B|=|C|=x$ and $A\cdot B\cap C=\varnothing$. This happens if and only if there exists $A,B\subset \F_q^\ast$ with $|A|=|B|=x$ and $ |A\cdot B|\le q-1-x.$\qedhere
\end{lemma}

\begin{proof}
The points $P_1(a),P_2(b)$ and $P_3(c)$ are collinear if and only if
\[
\det
\begin{pmatrix}
1&-a&0\\
0&1&b\\
1&0&c
\end{pmatrix}
=c-ab=0,
\]
that is, if and only if $c=ab$.

It follows that a general line contains three points of $S$, one on each side of the fundamental triangle, if and only if there exist
$a\in A$, $b\in B$, and $c\in C$ such that $c=ab$. Hence $S$ is a weak arc if and only if
$A\cdot B\cap C=\varnothing$. In particular, a balanced quasi-arc of parameter $x$ exists if and only if there are sets
$A,B,C\subseteq\F_q^\ast$ with
$|A|=|B|=|C|=x$
and $A\cdot B\cap C=\varnothing$.

Finally, for fixed $A,B$, a set $C\subseteq\F_q^\ast\setminus(A\cdot B)$ of size $x$ exists if and only if
\[
q-1-|A\cdot B|\ge x.
\]
\end{proof}

\begin{remark}
Taking
$A=B=\F_q^\ast$ and $C=\varnothing$ yields a weak arc with $2(q-1)$ points on two of the three sides of the fundamental triangle. It is easy to see that this set of size $2q+1$ is the largest possible weak arc whose points are contained on the sides of a triangle. Asking that the sets $A,B,C$ are non-empty, or that the set is a balanced quasi-arc makes the problem non-trivial.
\end{remark}

In what follows, we apply Kneser's theorem to the multiplicative group
$G=\F_q^\ast$.
For a subset $S\subseteq G$, we denote by $\Stab(S):=\{g\in G:gS=S\}$ its stabiliser in $G$. In particular,
\[
\Stab(A\cdot B)=\{g\in G:g(A\cdot B)=A\cdot B\}.
\].

\begin{lemma} \cite{Kneser}\label{lem:kneser}
Let $G$ be a finite abelian group and $A,B\subseteq G$ nonempty. Let
$H=\Stab(A\cdot B)$. Then 
\[
  |A\cdot B|\;\ge\;|A\!\cdot\! H|+|B\!\cdot\! H|-|H|\;\ge\;|A|+|B|-|H|.
\]
and $A\cdot B$ is a union of cosets of $H$. In particular $|H|\mid|A\cdot B|$.
\end{lemma}



\begin{proposition}\label{cor:odd}
Let $q$ be odd. The largest balanced quasi-arcs have parameter
\[
  a=\frac{q-1}{2},
\]
and this maximum is attained by taking $A=B$ to be the subgroup of squares in $\F_q^*$ and $C$ the set of
nonsquares in $\F_q^*$. Moreover, every balanced quasi-arc with parameter $\frac{q-1}{2}$ can be obtained via this construction, or by taking $A=B=C$ the non-squares in $\F_q^*$.
\end{proposition}

\begin{proof}
First let $A=B$ be the index-$2$ subgroup of order
$\frac{q-1}{2}$ in $\F_q^*$ (that is, the squares) and $C$ be the non-squares, then it is immediate that $A\cdot B\cap C=\emptyset$ and we have a balanced quasi-arc with parameter $\frac{q-1}{2}$. 
Now let $A,B,C$ be subsets of $\F_q^*$ with $x=|A|=|B|=|C|=\frac{q-1}{2}$ and $A\cdot B\cap C=\emptyset$. This immediately implies that $|A\cdot B|+|C|\leq q-1,$ so $|A\cdot B|\leq \frac{q-1}{2}.$ 

On the other hand, writing $H=\Stab(A\cdot B)$, Lemma~\ref{lem:kneser} gives
\[
  |A\cdot B|\ \ge\ |A|+|B|-|H|\ =\ (q-1)-|H|. 
\]
We find that
$  (q-1)-|H|\ \le\ \tfrac{q-1}{2}$, so 
  $|H|\ \ge\ \tfrac{q-1}{2}$.

Since $H$ is a subgroup of $\F_q^*$, we know that $|H|$ divides $q-1$. We first rule out the possibility that $|H|=q-1$ (that is, $H=\F_q^*$). In that case, for any element $s\in A\cdot B$, and every element $h\in \F_q^\ast$ we have that $sh\in A\cdot B$, which shows that $\F_q^*\subseteq A\cdot B$, a contradiction since $|A\cdot B|\leq \frac{q-1}{2}.$ Hence, $|H|=\tfrac{q-1}{2}$ and $H$ is the unique index-$2$ subgroup in $\F_q^*$ (given by the squares). Furthermore, by Lemma \ref{lem:kneser}, we know that $A\cdot B$ is a union
of $H$-cosets, which, since $|A\cdot B|\leq \frac{q-1}{2}$ implies that it is a single $H$-coset. Since $A\cdot B$ is a single $H$-coset, we can write $A\cdot B=tH$. Fixing any
$b_0\in B$, every $a\in A$ satisfies $ab_0\in A\cdot B=tH$, so
$A\subseteq tb_0^{-1}H$: the set $A$ lies in a single coset of $H$. As
$|A|=\frac{q-1}{2}=|H|$, $A$ is exactly that coset.
In the same way, we see that $B$ is a single
full coset of $H$ too. Since $C$ is the complement of $A\cdot B$, and $|C|=|A\cdot B|=\frac{q-1}{2}$, the elements of $C$ are those of the coset of $H$ different from the coset of $H$ given by $A\cdot B$.
We conclude that each of the sets $A,B,C$ are cosets of $H$. We can write those as $aH,bH,cH$ where $a,b,c$ are fixed elements in $A,B,C$ respectively, such that exactly two of $\{a,b,c\}$ are square and one is non-square, or all three of $a,b,c$ are non-square.
\end{proof}

When $q$ is even, $q-1$ is odd, so $\F_q^*$ does not have an index $2$ subgroup. The largest balanced quasi-arcs in this case can be attained by taking a {\em union} of several consecutive cosets of a proper subgroup. We first show the following lemma which is valid for general cyclic groups.

\begin{lemma}\label{lem:minprod}
Let $G$ be cyclic of order $n$ and $1\le a\le n$. Then
\[
  \min_{|A|=|B|=a}|A\cdot B|
  \;=\;\min_{d\mid n}\ d\bigl(2\lceil a/d\rceil-1\bigr).
\]
\end{lemma}

\begin{proof}
We write $F(a)=\min_{d\mid n}d\bigl(2\lceil a/d\rceil-1\bigr)$ for the right-hand
side. 

We first show that $\min_{|A|=|B|=a}|A\cdot B|
  \;\geq\;\min_{d\mid n}\ d\bigl(2\lceil a/d\rceil-1\bigr)$ by showing that for all sets $A,B$ with $|A|=|B|=a$, there exists a divisor $d$ of $n$ such that  $|A\cdot B|\ge d\bigl(2\lceil a/d\rceil-1\bigr)$.
So consider sets $A,B$ with $|A|=|B|=a$ and set $S=A\cdot B$, $H=\Stab(S)$ and $d=|H|$; then $d\mid n$ since $H$ is a subgroup of $G$. 
Since $H\cdot H=H$ we have $(A\cdot H)\cdot H=A\cdot H$, so $H$ is contained in
$K:=\Stab(A\cdot H)$; by Lemma \ref{lem:kneser} the set $A\cdot H$ is a disjoint
union of $K$-cosets, and as $H\le K$ each $K$-coset is in turn a disjoint union of
$H$-cosets, so $A\cdot H$ is a disjoint union of $H$-cosets, each of size exactly $d=|H|$. Say it consists of $k$ such cosets, so $|A\cdot H|=kd$. Since $e\in H$
we have $A\subseteq A\cdot H$, so these $k$ cosets cover all $a$ elements of $A$;
and each coset has size $d$, so $kd\ge a$. It follows that $k\ge\lceil a/d\rceil$ and
\[
  |A\cdot H|=kd\ \ge\ d\,\lceil a/d\rceil .
\]
Similarly, we find that $|B\cdot H|\ge\ d\,\lceil a/d\rceil$. Lemma \ref{lem:kneser} now gives $  |A\cdot B|\ \ge\ 2\,d\lceil a/d\rceil-d\ =\ d\bigl(2\lceil a/d\rceil-1\bigr).$

We now show that $\min_{|A|=|B|=a}|A\cdot B|
\leq\min_{d\mid n}\ d\bigl(2\lceil a/d\rceil-1\bigr)$ by showing that for all $1\leq a\leq n$, and for each divisor $d\mid n$ with $2\lceil a/d\rceil-1\le
n/d$, we can construct a pair $A,B$ with $|A|=|B|=a$ and $|A\cdot B|\le d\bigl(2\lceil a/d\rceil-1\bigr)$. Adding the condition that $2\lceil a/d\rceil-1\le
n/d$ is not a restriction, since for a divisor $d$ with $2\lceil a/d\rceil-1>
n/d$, all sets with $|A|=|B|=a$ satisfy $|A\cdot B|\leq d\bigl(2\lceil a/d\rceil-1\bigr)$ as $\bigl(2\lceil a/d\rceil-1\bigr)>n$.

So fix $1\leq a\leq n$ and a divisor $d$ of $n$ with with $2\lceil a/d\rceil-1\le
n/d$. Let $H=\langle g^{n/d}\rangle$ be the unique subgroup of order $d$, where $g$ is a generator of $G$. The quotient $G/H$ is cyclic of order $N:=n/d$.
Writing $m=\lceil a/d\rceil$, we know that $2m-1\le N$ and from
$m=\lceil a/d\rceil$ we have $d(m-1)<a\le dm$. We see that $r:=a-(m-1)d$ satisfies
$1\le r\le d$ and we can take
\[
  A=B=(g^0H)\cup\cdots\cup(g^{m-2}H)\ \cup\ R,
\]
where $R\subseteq g^{m-1}H$ is any $r$-element subset. Then $|A|=(m-1)d+r=a$, and
the powers of $g$ appearing as coset leaders in the set $A$ are exactly $I=\{0,1,\dots,m-1\}$. Therefore, every element in $A\cdot B$ is contained in some coset with coset leader
in \[
  I+I=\{0,1,\dots,2m-2\}\pmod N .
\]
Since $2m-2\le N-1$, this set has exactly $2m-1$ elements, and $A\cdot B$ lies in the union of $2m-1$ cosets of $H$, giving
$|A\cdot B|\le d(2m-1)=d\left(2\lceil \frac{a}{d}\rceil-1\right)$. 
\end{proof}

\begin{theorem}\label{thm:main}
Let $q$ be a prime power and $n=q-1$. Let $S$ be a balanced quasi-arc of parameter
$a$ in $\PG(2,q)$. Then   $a\ \le\ a_{\max}$ where
\[
  a_{\max}\ =\ \max_{\substack{d\mid n,\ m\ge 1\\ d(3m-2)<n}}
  \ \min\bigl(dm,\ n-d(2m-1)\bigr).
\]
Moreover the bound is sharp. A balanced quasi-arc of parameter
$a_{\max}$ can be obtained as follows: let $(d,m)$ be the parameters defining $a_{\max}$
and let $H\le \F_q^\ast$ be the subgroup of order $d$ with $g$
a generator of $\F_q^\ast$. Take $A=B$ to consist of the $m-1$ cosets
$g^0H,\dots,g^{m-2}H$ together with an arbitrary $\bigl(a_{\max}-(m-1)d\bigr)$-element
subset of the $m$-th coset $g^{m-1}H$. Then
$|A\cdot B|=d(2m-1)\le n-a_{\max}$, and $C$ may be taken to be any $a_{\max}$-subset of
$\F_q^\ast\setminus(A\cdot B)$.
\end{theorem}

\begin{proof} Let $S$ be a balanced quasi-arc of parameter
$a$ in $\PG(2,q)$. By Lemma~\ref{prop:reduction}, $\min_{|A|=|B|=a}|A\cdot B|\le n-a$, which by
Lemma~\ref{lem:minprod} is equivalent to the existence of a divisor $d$ of $n$ such that 
\begin{align}d\bigl(2\lceil a/d\rceil-1\bigr)\ \le\ n-a \label{1}.\end{align}
Now set $m=\lceil a/d\rceil$, so that
  $d(m-1)<a\le dm$, and \eqref{1} reads $d(2m-1)\le n-a$, so $a\le n-d(2m-1)$. We see that
\begin{align}
  a\ \le\ \min\bigl(dm,\ n-d(2m-1)\bigr). \label{2}
\end{align}
Furthermore, the lower bound $a>d(m-1)$ forces $\min\bigl(dm,\ n-d(2m-1)\bigr)>d(m-1)$. Since $dm>d(m-1)$, this holds if
and only if $
  n-d(2m-1)>d(m-1)$, that is, if and only if $
  n>d(3m-2)$. 
We now show that the bound is sharp. Let $(d,m)$ realise the maximum $a_{\max}$; we have $d(3m-2)<n$. Let $H\le \F_q^\ast$ be the subgroup of
order $d$ with generator $g$ and index $N=n/d$. We have that $r:=a-(m-1)d$ satisfies $1\le r\le d$. As in the proof of Lemma \ref{lem:minprod}, we can
 set
$
  A=B=(g^0H)\cup\cdots\cup(g^{m-2}H)\ \cup\ R,\qquad R\subseteq g^{m-1}H$ where $|R|=r$,
and find that $  |A\cdot B|\leq d(2m-1)\leq n-a$, so we can take a subset $C$ of size $a$ disjoint from $A\cdot B$ in $\F_q^\ast $ to construct a balanced quasi-arc of parameter $a=a_{\max}$.
\end{proof}

We can use the previous theorem to recover the bound for odd order $q$ established in Proposition \ref{cor:odd}. However, unlike Proposition \ref{cor:odd} the following corollary does not characterise the largest balanced quasi-arcs.

\begin{corollary}
Let $q$ be odd, so $n=q-1$ is even. Then
\[
  a_{\max}(q)=\frac{q-1}{2},
\]
attained at $(d,m)=\bigl(\tfrac{q-1}{2},\,1\bigr)$. This can be constructed by taking $A=B=H$ the
index-$2$ subgroup (the squares) and $C=\F_q^\ast\setminus H$.
\end{corollary}

\begin{proof}
For any admissible pair $(d,m)$ the elementary inequality $\min(x,y)\le\frac{x+y}{2}$ gives
\[
  \min\bigl(dm,\ n-d(2m-1)\bigr)
  \ \le\ \frac{dm+\bigl(n-d(2m-1)\bigr)}{2}
  \ =\ \frac{n-d(m-1)}{2}\ \le\ \frac{n}{2},
\]
so $a_{\max}(q)\le\frac{n}{2}=\frac{q-1}{2}$.

Since $n=q-1$ is even, $d=\frac{q-1}{2}$ is a
divisor of $n$. The construction of Theorem \ref{thm:main} for $m=1,d=n/2$ then reduces to one described in the statement of the corollary.
\end{proof}

\begin{remark}
If $q$ is even, so $n=q-1$ is odd, there is no index-$2$ subgroup, and one might guess the extremal $A=B$ is found by using a single subgroup, giving $a_{\max}=\frac{q-1}{p_{\min}}$ with
$p_{\min}$ the least prime factor of $q-1$. However, this is in general not optimal: taking $m\ge 2$
cosets of a different subgroup may do strictly better. For example, consider $q=16$, $n=15$: the single-subgroup construction using the index $3$ subgroup gives a balanced quasi-arc with $a=15/3=5$; but using two cosets of the index $5$-subgroup $d=3$, $m=2$
  yields $a=\min(6,\,15-9)=6$ which, given Theorem \ref{thm:main}, is best possible.
\end{remark}

\subsection{Weak arcs and weak caps in three dimensions}\label{subs:higher}
\subsubsection{A weak cap contained in the edges of the fundamental tetrahedron}
In this section, we construct a weak cap in $\PG(3,q)$, contained on the edges of the fundamental tetrahedron. The point set is not a weak arc since there are still planes with four points of this set. In Section \ref{subsection:subset}, we will show how to take a subset of this point set that is a balanced quasi-arc.

\begin{proposition}\label{prop:cap}
Let $H$ be a multiplicative 
subgroup of $\F_q^*$, and let
\[
K:=\F_q^*\setminus H.
\]
In $\PG(3,q)$, let
\[
\begin{aligned}
\mathcal S_{12}&=\{(1,-a,0,0):a\in H\},\\
\mathcal S_{13}&=\{(1,0,b,0):b\in K\},\\
\mathcal S_{14}&=\{(1,0,0,c):c\in H\},\\
\mathcal S_{23}&=\{(0,1,d,0):d\in H\},\\
\mathcal S_{24}&=\{(0,1,0,e):e\in K\},\\
\mathcal S_{34}&=\{(0,0,1,-f):f\in H\},
\end{aligned}
\]
and let
\[
\mathcal S=
\{P_1,P_2,P_3,P_4\}
\cup
\bigcup_{1\le i<j\le 4}\mathcal S_{ij},
\]
where
\[
P_1=(1,0,0,0),\quad
P_2=(0,1,0,0),\quad
P_3=(0,0,1,0),\quad
P_4=(0,0,0,1).
\]
Then  $\mathcal S$ is a weak cap of size $2(q+1+|H|).$
\end{proposition}

\begin{proof}
Note that if a line $\ell$ not containing a fundamental point contains three points of $\mathcal S$, then it meets three different edges of the tetrahedron. All sets of three edges of the tetrahedron contain at least two edges on the same face, and hence, we only need to consider lines contained in one of the four faces of the tetrahedron.

\textbf{Case 1:} $\ell$ meets the edges $P_1P_2$, $P_1P_3$, and $P_2P_3$. Let
$
X=(1,-a,0,0)\in\mathcal S_{12}$, 
$Y=(1,0,b,0)\in\mathcal S_{13}$,
$Z=(0,1,d,0)\in\mathcal S_{23}$, where $a,d\in H$ and $b\in K$.

These points are collinear if and only if $ad=b$.
Since $a,d\in H$, and $H$ is a multiplicative subgroup, we have $ad\in H$, contradicting $b\in K$.

\textbf{Case 2:} $\ell$ meets the edges $P_1P_2$, $P_1P_4$, and $P_2P_4$.
Similarly, we obtain the condition $ae=c$, a contradiction since $a\in H$, $e\in K$, $c\in H$.

\textbf{Case 3:} $\ell$ meets the edges $P_1P_3$, $P_1P_4$, and $P_3P_4$.
Now we find  $bf=c$, a contradiction since $b\in K$ and $f\in H$, $c\in H$.

\textbf{Case 4:} $\ell$ meets the edges $P_2P_3$, $P_2P_4$, and $P_3P_4$.
Now we find $df=e$, a contradiction since $d,f\in H$, $e\in K$.
Finally, the size of the set is given by $4|H|+2(q-1-|H|)+4=2(q+1+|H|).$
\end{proof}

If $q$ is odd, we can take $H$ to be the index $2$ subgroup of $\F_q^*$ and find the following:
\begin{corollary}
   If $q$ is odd, there exists a weak cap of size $3q+1$ whose points lie on the edges of the fundamental tetrahedron and every edge contains $\frac{q-1}{2}$ points different from the fundamental points.
\end{corollary}

\begin{remark}
When $q$ is odd, there is also a more symmetric choice. Let $H$ be the
subgroup of squares of $\F_q^\ast$ and let $K=\F_q^\ast\setminus H$ be
the set of nonsquares. Since $K\cdot K=H$, we may take
\begin{align*}
\mathcal S_{12}&=\{(1,-a,0,0):a\in K\},\\
\mathcal S_{13}&=\{(1,0,b,0):b\in K\},\\
\mathcal S_{14}&=\{(1,0,0,c):c\in K\},\\
\mathcal S_{23}&=\{(0,1,d,0):d\in K\},\\
\mathcal S_{24}&=\{(0,1,0,e):e\in K\},\\
\mathcal S_{34}&=\{(0,0,1,-f):f\in K\}.
\end{align*}
Indeed, each of the possible collinearity conditions on a face has the
form that the product of two elements of $K$ equals another element of
$K$. This is impossible, since $K\cdot K=H$ and $H\cap K=\varnothing$.
Hence, this choice also gives a weak cap of size $3q+1$.
\end{remark}


\subsubsection{A quasi-arc as subset of the weak cap}\label{subsection:subset}

Let $q$ be a prime power with $q\equiv 1\pmod 4$, so that the multiplicative
group $\mathbb{F}_q^*$ contains a subgroup $G$ of index $4$. Let $H$ be the subgroup of nonzero squares of $\F_q$, then  $G$ is its unique
index-$2$ subgroup. Choose a generator $\omega$ of the cyclic group
$\mathbb{F}_q^*/G\cong\mathbb{Z}/4\mathbb{Z}$, then the four cosets of $G$ in $\mathbb{F}_q^*$ are $
   G, \omega G, \omega^2G, \omega^3G$, and $H=G\cup \omega^2G$, $K=\F_q^*\setminus H =\omega G\cup \omega^3G.$

\begin{proposition}\label{prop:weakarc3D} Define

\begin{align*}
\mathcal S_{12}&=\{(1,-a,0,0):a\in G\},\\
\mathcal S_{13}&=\{(1,0,b,0):b\in \omega G\},\\
\mathcal S_{14}&=\{(1,0,0,c):c\in \omega ^2G\},\\
\mathcal S_{23}&=\{(0,1,d,0):d\in G\},\\
\mathcal S_{24}&=\{(0,1,0,e):e\in \omega^3G\},\\
\mathcal S_{34}&=\{(0,0,1,-f):f\in G\},
\end{align*}

and let
\[
\mathcal S=
\{P_1,P_2,P_3,P_4\}
\cup
\bigcup_{1\le i<j\le 4}\mathcal S_{ij},
\]
where
\[
P_1=(1,0,0,0),\quad
P_2=(0,1,0,0),\quad
P_3=(0,0,1,0),\quad
P_4=(0,0,0,1).
\]
Then  $\mathcal S$ is a weak arc of size $\frac{3q+5}{2}$.
\end{proposition}

\begin{proof} Since $\mathcal{S}$ is a subset of the weak cap we constructed in Proposition \ref{prop:cap}, no general plane can meet the set in $3$ collinear points. This means that a general plane can meet each of the faces of the fundamental tetrahedron in at most $2$ points, and each such point lies on $2$ faces. It follows that a general plane can meet the set $\mathcal{S}$ in at most $4$ points. Furthermore, no plane which meets $3$ edges which contain a common fundamental point can meet the set in $4$ points.
This reduces the possibilities for planes with $4$ points to the following: the plane meets the sets
\begin{itemize}
\item $\mathcal{S}_{12},\mathcal{S}_{13},\mathcal{S}_{24},\mathcal{S}_{34}$, or
\item $\mathcal{S}_{12},\mathcal{S}_{14},\mathcal{S}_{23},\mathcal{S}_{34}$, or
\item $\mathcal{S}_{13},\mathcal{S}_{14},\mathcal{S}_{23},\mathcal{S}_{24}.$
\end{itemize}

The conditions on $a,b,c,d,e,f$ that arise from expressing that four points of those sets are coplanar are respectively:
$bf=ae$, $c=adf$, $be=cd$.
For the first condition, we see that the left hand side is contained in $\omega G$ whereas the right hand side is contained in $\omega^3 G$. Similarly, for the second condition, the left hand side is in $\omega^2 G$ and the right hand side in $G$. For the third condition, the left hand side is in $\omega^4G=G$ and the right hand side in $\omega^3G$. All of those lead to a contradiction.
\end{proof}

\section{Weak arcs for the DNA storage problem}\label{sec:arcDNA}
The goal of this section is to give some evidence to the heuristic that balanced quasi-arcs are performing well for the DNA storage problem. Recall that we want to minimise the maximum value $M$ of $\E[\tau_{P_i}(\Gscr)]$ for the fundamental points. While it is not true that balanced quasi-arcs are those point sets reaching the minimum for $M$ (see Remark \ref{rem:k3}), we will show that starting from a point set which is a balanced quasi-arc, and adding three points, if it is possible to add one point to each side and still be a weak arc, this will be better than adding the three points off the sides of the triangle. 
We will also show that a balanced quasi-arc does better compared to an unbalanced weak arc contained in the sides of the fundamental triangle, and that a balanced quasi-arc of parameter $(q-1)/2$ is optimal amongst balanced sets contained in the sides of the triangle. 

We start with an elementary but useful rewriting of the value $\E[\tau_{P_i}(\Gscr)]$.

\subsection{An alternative expression for \texorpdfstring{$\E[\tau_{P_i}(\Gscr)]$}{E[tau	extunderscore{P	extsubscript{i}}(G)]}}\label{subs:alt}

We have the following expression for $\E[\tau_{P_i}(\Gscr)]$, where the size of the multiset $\Gscr$ is $n$.
\begin{equation}\label{eq:alpha}
\E[\tau_{P_i}(\Gscr)] \;=\; n H_n \;-\; \sum_{s=1}^{n-1}\frac{\alpha_{P_i}(\G,s)}{\binom{n-1}{s}},
\end{equation}
where $H_n = 1+\tfrac{1}{2}+\cdots+\tfrac{1}{n}$ is the $n$th harmonic number and
$\alpha_{P_i}(\G,s)$ counts the $s$-subsets of $\G$ whose projective span
contains $P_i$.
More precisely, for a fundamental point $P_i$ and $1\le s\le n$, let
\[
\alpha_{P_i}(\Gscr,s)
:=
|\left\{
A\subseteq\Gscr:
|A|=s,\;
P_i\in\langle A\rangle
\right\}|.
\] In what follows, the multiset $\Gscr$ will often be clear from the context, so we also write $\alpha_{P_i}(s)$ for $\alpha_{P_i}(\Gscr,s)$ and $\E[\tau_{P_i}]$ for $\E[\tau_{P_i}(\Gscr)].$

\begin{lemma}\label{lem:alfabeta}\cite[Appendix A]{WangYaakobi2026}
We have that
\begin{equation}\label{eq:beta}
\E[\tau_{P_i}(\Gscr)] \;=\; 1 \;+\; \sum_{s=1}^{n-1}\frac{\beta_{P_i}(s)}{\binom{n-1}{s}},
\end{equation}
where
\[
\beta_{P_i}(s)\;:=\;\binom{n}{s}-\alpha_{P_i}(s)
\]
counts the $s$-subsets of $\Gscr$ whose span does not contain $P_i$. \end{lemma}


Later on, we will write
\[
\E[\tau_{P_1}]\;=\;1 \;+\;\frac{\beta_{P_1}(1)}{\binom{n-1}{1}}+\frac{\beta_{P_1}(2)}{\binom{n-1}{2}}+\sum_{s=3}^{n-1}\frac{\beta_{P_1}(s)}{\binom{n-1}{s}},
\]
and in the simplification of $\sum_{s=3}^{n-1}\frac{\beta_{P_1}(s)}{\binom{n-1}{s}}$ we will use the following.

\begin{lemma}\label{lem:UT}
Let $n\ge 2$ and $3\le m\le n-1$. Define
\begin{align}
T(m):=\sum_{s=3}^{n-1}\frac{\binom{m}{s}}{\binom{n-1}{s}}. \label{eq:Tdef}
\end{align}
Then
\[
T(m)=\frac{m(m-1)(m-2)}{(n-1)(n-2)(n-m)}.
\]
\end{lemma}

\begin{proof}
Write $f(s):=\frac{\binom{m}{s}}{\binom{n-1}{s}}$, so that $T(m)=\sum_{s=3}^{m}f(s)$.

Define, for $3\le s\le m$,
\[
g(s)=\frac{n-s}{n-m}\cdot f(s),
\]
and set $g(m+1)=0$.

We claim that
$f(s)=g(s)-g(s+1)$ 
for every $3\le s\le m$. For $s=m$, this is clear. For $3\le s\le m-1$, we have
\begin{align*}
g(s)-g(s+1)
&=\frac{n-s}{n-m}f(s)-\frac{n-s-1}{n-m}f(s+1) \\
&=\frac{f(s)}{n-m}\left((n-s)-(n-s-1)\frac{f(s+1)}{f(s)}\right).
\end{align*}
Since
\[
\frac{f(s+1)}{f(s)}
=\frac{\binom{m}{s+1}\big/\binom{m}{s}}
{\binom{n-1}{s+1}\big/\binom{n-1}{s}}
=\frac{m-s}{n-1-s},
\]
we obtain
\[
g(s)-g(s+1)
=\frac{f(s)}{n-m}\left((n-s)-(n-s-1)\cdot\frac{m-s}{n-1-s}\right)
=\frac{f(s)}{n-m}\left(n-m\right)=f(s),
\]
as claimed. Therefore,
\[
T(m)=\sum_{s=3}^{m}f(s)
=\sum_{s=3}^{m}\bigl(g(s)-g(s+1)\bigr)
=g(3)-g(m+1)
=g(3),
\]
where\[
g(3)=\frac{n-3}{n-m}\cdot \frac{\binom{m}{3}}{\binom{n-1}{3}}
=\frac{m(m-1)(m-2)}{(n-1)(n-2)(n-m)}.
\] \end{proof}

\subsection{Balanced quasi-arcs are a good choice}\label{subs:opt}
While it is not possible to show that $M$ is optimal when a point set is a quasi-arc (see Remark \ref{rem:k3}), the following two propositions support the heuristic that these point sets are a good choice.

In the next proposition, we show that adding one point to each side of a balanced quasi-arc is better than a generic off-side extension of a balanced quasi-arc.

\begin{proposition}
\label{lem:three-point-side-extension}
Let $P_1,P_2,P_3$ be the fundamental points, and let $
\mathcal{S}\subseteq P_1P_2\cup P_1P_3\cup P_2P_3$
be a balanced quasi-arc of parameter $a$ of size $3+3a$. 
Assume that there are points
$Y_1\in P_2P_3,
Y_2\in P_1P_3,
Y_3\in P_1P_2$,
not belonging to $\mathcal{S}$, such that $\mathcal{S^+}:=\mathcal{S}\cup\{Y_1,Y_2,Y_3\}$ is a balanced quasi-arc with parameter $a+1$.  Now let $X_1,X_2,X_3$ be three points not on the sides of the fundamental triangle, and put
\[
\mathcal{S^{\circ}}:=\mathcal{S}\cup\{X_1,X_2,X_3\}.
\]
Assume that $\mathcal{S}^{\circ}$ is a weak arc and that the points $X_1,X_2,X_3$ are generic: {the lines $X_1X_2,X_1X_3,X_2X_3$ do not contain any of the points $P_1,P_2,P_3$.}  Then
\[
M(\mathcal{S}^+) < M(\mathcal{S}^{\circ}).
\]
\end{proposition}

\begin{proof}
Write $n=|\mathcal{S}^{+}|=|\mathcal{S}^{\circ}|=3a+6$ and as in Lemma \ref{lem:alfabeta}, write $\beta_{P_i}(\Gscr,s)=\binom{n}{s}-\alpha_{P_i}(\Gscr,s)$. Using
\eqref{eq:beta} we find that
\[
\Delta_i
:=\E[\tau_{P_i}(\mathcal{S}^{\circ})]-\E[\tau_{P_i}(\mathcal{S}^{+})]
=\sum_{s=1}^{n-1}
    \frac{\beta_{P_i}(\mathcal{S}^{\circ},s)-\beta_{P_i}(\mathcal{S}^{+},s)}{\binom{n-1}{s}} .\]

We will only consider $i=1$. Since $\mathcal{S}^{+}$ is a balanced quasi-arc,  $M(\mathcal{S}^{+})=\E[\tau_{P_1}(\mathcal{S}^{+}]$, and trivially
$M(\mathcal{S}^{\circ})\ge\E[\tau_{P_1}(\mathcal{S}^{\circ})]$, so it suffices to
prove $\Delta_1>0$.

{Note that in either configuration, all subsets of at least $3$ points whose span does not contain $P_1$ are contained in $P_2P_3$; this follows from the fact that $\mathcal{S}^+$ is a balanced quasi-arc, and the points $X_1,X_2,X_3$ of the weak arc $\mathcal{S}^\circ$ span lines not containing $P_1,P_2,P_3$. Since the number of points of $\mathcal{S}^+$ on $P_2P_3$ is $a+3$, and of $\mathcal{S}^{\circ}$ on $P_2P_3$ is $a+2$ it follows that for $s\geq 3$, $\beta_{P_1}(\mathcal{S}^+,s)=\binom{a+3}{s}$ and $\beta_{P_1}(\mathcal{S}^{\circ},s)=\binom{a+2}{s}$. Hence, for $s\geq 3$, $\beta_{P_1}(\mathcal{S}^{\circ},s)-\beta_{P_1}(\mathcal{S}^{+},s)=\binom{a+2}{s}-\binom{a+3}{s}=-\binom{a+2}{s-1}.$
The number of points $\neq P_1$ on $P_1P_2$ and on $P_1P_3$ in $\mathcal{S}^+$ is equal to $a+2$, and in $\mathcal{S}^\circ$ is equal to $a+1$, so we obtain
$\beta_{P_1}(\mathcal{S}^+,2)=\binom{n-1}{2}-2\binom{a+2}{2}$ and $\beta_{P_1}(\mathcal{S}^\circ,2)=\binom{n-1}{2}-2\binom{a+1}{2}$. Hence, 
$\beta_{P_1}(\mathcal{S}^\circ,2)-\beta_{P_1}(\mathcal{S}^+,2)
=-2\binom{a+1}{2}+2\binom{a+2}{2}=2(a+1).$
Finally, $\beta_{P_1}(\mathcal{S}^+,1)=\beta_{P_1}(\mathcal{S}^\circ,1)=n-1.$}
Since $n-1=3a+5$, we obtain
\[
\Delta_1
=\frac{2(a+1)}{\binom{3a+5}{2}}
 \;-\;\sum_{s=3}^{a+3}\frac{\binom{a+2}{s-1}}{\binom{3a+5}{s}} .
\]

We can rewrite $S_a=\sum_{s=3}^{a+3}\frac{\binom{a+2}{s-1}}{\binom{3a+5}{s}}$ using $\binom{a+2}{i-1}=\binom{a+3}{i}-\binom{a+2}{i}$ and find that $
S_a=T(a+3)-T(a+2)$ for $T$ as defined in \eqref{eq:Tdef}.

Now
\[
T(a+3)=\frac{(a+3)(a+2)(a+1)}{(3a+5)(3a+4)(2a+3)},
\qquad
T(a+2)=\frac{(a+2)(a+1)a}{(3a+5)(3a+4)(2a+4)}.
\]
We see that our claim that $\Delta_1>0$ is equivalent to
\[
T(a+3)-T(a+2)\;<\;\frac{2(a+1)}{\binom{3a+5}{2}}
\;=\;\frac{4(a+1)}{(3a+4)(3a+5)},
\]
which is equivalent to
\[
\frac{a+3}{2a+3}-\frac{a}{2a+4}\;<\;4.
\]
It is easy to see that this inequality holds.
\end{proof}

In the same spirit, we will now show that for all weak arcs contained on the sides of the fundamental triangle, the balanced quasi-arc has the lowest value of $M$.

\begin{proposition}
\label{prop:balbetter}
Let $
\Gscr^{a_1,a_2,a_3}\subseteq P_2P_3\cup P_1P_3\cup P_1P_2$ 
denote a weak arc of size $n=3+3a$ containing $P_1,P_2,P_3$ where $a_i$ is the number of non-fundamental points of
$\Gscr$ on the side $P_jP_k$. Thus $a_1+a_2+a_3=3a$.
 Then
\[
M(\Gscr^{a,a,a})\le M(\Gscr^{a_1,a_2,a_3}).
\]
With strict inequality unless
\[
a_1=a_2=a_3=a.
\]
\end{proposition}

\begin{proof}
The averaged recovery expectation $\overline{\E}(\Gscr)$ over the three fundamental points satisfies
$\overline{\E}(\Gscr)\le M(\Gscr)$, and in the symmetric case $a_1=a_2=a_3$ all three values
coincide, so $\overline{E}(\Gscr^{a,a,a})=M(\Gscr^{a,a,a})$. Hence
\[
M(\Gscr^{a,a,a})=\overline{\E}(\Gscr^{a,a,a})
\le\overline{\E}(\Gscr^{a_1,a_2,a_3})
\le M(\Gscr^{a_1,a_2,a_3}),
\]
provided $\overline{\E}$ is indeed minimised at $(a,a,a)$; it suffices to prove the latter.

As before, for a point $P$ and integer $s$, let
$\beta_P(\Gscr^{a_1,a_2,a_3},s)$ be the number of $s$-subsets of $\Gscr^{a_1,a_2,a_3}$ whose span does not contain $P$. We have $
\beta_{P_i}(\Gscr^{a_1,a_2,a_3},1)=n-1$.
Since there there are $\binom{a_j+1}{2}+\binom{a_k+1}{2}$ pairs of points lying on $P_iP_j$ and $P_iP_k$ we have
\[
\beta_{P_i}(\Gscr^{a_1,a_2,a_3},2)=\binom{n-1}{2}-\binom{a_j+1}{2}-\binom{a_k+1}{2}.
\]
The only subsets of size $s\ge3$ whose span does not contain $P_i$ are those contained in $P_jP_k$. So for $s\geq 3$, $\beta_{P_i}(\Gscr^{a_1,a_2,a_3},s)=\binom{a_i+2}{s}$.
This gives
\[
\E[\tau_{P_i}(\Gscr^{a_1,a_2,a_3})]
=1+1+1-\frac{\binom{a_j+1}{2}+\binom{a_k+1}{2}}{\binom{n-1}{2}}
+\sum_{s=3}^{n-1}\frac{\binom{a_i+2}{s}}{\binom{n-1}{s}}\\=3-\frac{\binom{a_j+1}{2}+\binom{a_k+1}{2}}{\binom{n-1}{2}}
+T(a_i+2).
\]

When taking the average over $i=1,2,3$, each term $\binom{a_r+1}{2}$ occurs twice in the sum, so we
obtain
\[
\overline{\E}(\Gscr^{a_1,a_2,a_3})=3+\tfrac13\sum_{r=1}^{3}\phi(a_r)
\text{, where }
\phi(t)=T(t+2)-\frac{2\binom{t+1}{2}}{\binom{n-1}{2}}.
\]
The problem reduces to showing that
$\sum_{r}\phi(a_r)\ge3\phi(a)$ when $a_1+a_2+a_3=3a$, and the inequality is strict unless all $a_1=a_2=a_3=a$.
A direct calculation shows that it is true (even though $\phi(t)$ is not convex), see Lemma \ref{lem:phi} in the Appendix.

\end{proof}
Finally, we show that, for odd $q$, balanced quasi-arcs minimise the recovery expectation among all balanced point sets contained in the sides of the fundamental triangle.

\begin{theorem}
\label{thm:optimal-simple-side-configuration}
Let $q$ be odd and let $P_1,P_2,P_3$ be fundamental points. Consider point
sets $\mathcal S$ (without multiplicities) with the following properties:
\begin{enumerate}
\item $\mathcal S$ is contained in the union of the three sides
      $P_1P_2\cup P_1P_3\cup P_2P_3$;
\item $P_1,P_2,P_3\in\mathcal S$;
\item each side contains exactly $a$ points of $\mathcal S$ other than
      its two fundamental points, for some $1\leq a\leq q-1$.
\end{enumerate}
Then $\mathbb E[\tau_{P_i}]$ is independent of the choice of the fundamental point $P_i$ and is
minimised when
\[
a=\frac{q-1}{2}
\]
and $\mathcal S$ is a balanced quasi-arc.
In that case, we have that for $i=1,2,3$,
\begin{equation}
\mathbb E[\tau_{P_i}]
 =\frac{3(17q^3+q^2-q-1)}{2q(3q-1)(3q+1)}.
\end{equation}
In particular,
\[
\lim_{q\to\infty}\mathbb E[\tau_{P_i}]=\frac{17}{6}\approx 2.8333,
\qquad
\lim_{q\to\infty}\frac{\mathbb E[\tau_{P_i}]}{3}=\frac{17}{18}\approx 0.9444.
\]
\end{theorem}

\begin{proof}
Fix a point set $\mathcal S$ satisfying the hypotheses, and let $t$ be
the number of $3$-secant lines to $\mathcal{S}$, not containing a fundamental point. Put $n=3a+3.$
The number $\beta_{P_i}(s)$ (the number of $s$-subsets of
$\mathcal S\setminus\{P_i\}$ whose span does not contain
$P_i$) will only depend on $a$ and $t$, so is independent of $i$.  We calculate them for $P_1$.

Clearly, $\beta_{P_1}(1)=n-1.$
For $s\geq2$, the lines spanned by points of $\mathcal{S}$, not through $P_1$ are as follows:
\begin{itemize}
\item the opposite side $P_2P_3$ which contains $a+2$ points of
$\mathcal S$;
\item the $2a$ two-secants obtained by
joining $P_2$ or $P_3$ to a selected non-fundamental point on the
opposite side through $P_1$;
\item the lines not through any fundamental point. There are $3a^2$ pairs consisting of selected non-fundamental points on
two different sides.  Each of those pairs spans a line not containing $P_1$. Furthermore, each of the $t$ $3$-secant lines determines $3$ such pairs, so there are $t$ $3$-secants, and $3a^2-3t$ $2$-secant lines not containing a fundamental point, spanned by points of $\mathcal{S}$.
\end{itemize}
It follows that
\begin{align*}
\beta_{P_1}(2)
  =\binom{a+2}{2}+2a+3a^2,\quad
\beta_{P_1}(3)
 =\binom{a+2}{3}+t,\quad
\beta_{P_1}(s)
 =\binom{a+2}{s},\qquad s\geq4.
\end{align*}

Using
\[
\mathbb E[\tau_{P_1}]
 =1+\frac{\beta_{P_1}(1)}{\binom{n-1}{1}}
   +\frac{\beta_{P_1}(2)}{\binom{n-1}{2}}
   +\sum_{s=3}^{n-1}\frac{\beta_{P_1}(s)}{\binom{n-1}{s}}
\]
together with the identity of Lemma~\ref{lem:UT} (with $n=3a+3$), 
we obtain
\begin{align}
\mathbb E[\tau_{P_1}]={}&2+\frac{\binom{a+2}{2}+3a^2+2a}{\binom{3a+2}{2}}+\frac{t\binom{3}{3}}{\binom{3a+2}{3}}
     +\sum_{s=3}^{n-1}\frac{\binom{a+2}{s}}{\binom{n-1}{s}}
    \nonumber\\
={}&2+\frac{\binom{a+2}{2}+3a^2+2a}{\binom{3a+2}{2}}
     +\frac{t}{\binom{3a+2}{3}}+T(a+2)\nonumber\\
={}&F(a)+\frac{t}{\binom{3a+2}{3}},
\label{eq:expectation-a-t}
\end{align}
where we set
\begin{equation}
\label{eq:F-a-definition}
F(a)=
\frac{3(17a^3+26a^2+13a+2)}
     {(2a+1)(3a+1)(3a+2)}.
\end{equation}

We now minimize \eqref{eq:expectation-a-t}.  Note that $F$ is
strictly decreasing in $a$, since
\begin{equation}
\label{eq:F-decreasing}
F(a+1)-F(a)
=-\frac{3(a+1)(9a^3+35a^2+32a+4)}
{(2a+1)(2a+3)(3a+1)(3a+2)(3a+4)(3a+5)}<0.
\end{equation}
Put $h=\frac{q-1}{2}$.  If $a\leq h$ it follows that (since $t\geq 0$),
\[
\mathbb E[\tau_{P_1}]\geq F(a)\geq F(h),
\]
with equality only if $a=h$ and $t=0$.

We now exclude $a>h$.  Let $P$ be a non-fundamental point in $\mathcal{S}$ on the side $P_1P_2$. There are $q-1$ lines through $P$ not containing $P_1,P_2,P_3$ and there are $a$ points of $\mathcal{S}$ on $P_1P_3$ and on $P_2P_3$. Hence, the number of general $3$-secants through $P$, which are those containing a point of $\mathcal{S}$ on both $P_1P_3$ and $P_2P_3$ is at least $a+a-(q-1)=2(a-h)$. Since this holds for all $a$ non-fundamental points of $\mathcal{S}$ on the side $P_1P_2$, it follows that $t\geq 2a(a-h).$

It then follows from \eqref{eq:expectation-a-t} that
\begin{align*}
\mathbb E[\tau_{P_1}]
&\geq F(a)+\frac{a(2a-2h)}{\binom{3a+2}{3}}\\
&=F(a)+\frac{4(a-h)}{(3a+1)(3a+2)}.
\end{align*}
We claim that if $a>h$, this expression is always strictly larger than $F(h)$.  We can calculate
\begin{equation}
\label{eq:large-a-difference}
F(a)+\frac{4(a-h)}{(3a+1)(3a+2)}-F(h)
=\frac{(a-h)R(a,h)}
{(2a+1)(3a+1)(3a+2)(2h+1)(3h+1)(3h+2)},
\end{equation}
where
\begin{align*}
R(a,h)={}&a\bigl(144h^3+177h^2+59h+10
                 -a(27h^2+39h+6)\bigr)\\
&\quad+72h^3+102h^2+46h+8.
\end{align*}
Since $a\leq q-1=2h$, we have
\begin{align*}
144h^3+177h^2+59h+10-a(27h^2+39h+6)
&\geq 144h^3+177h^2+59h+10\\
&\qquad-2h(27h^2+39h+6)\\
&=90h^3+99h^2+47h+10>0.
\end{align*}
Thus $R(a,h)>0$ which proves our claim.
Substituting $h=(q-1)/2$ into \eqref{eq:F-a-definition} gives the expression of the statement.

\end{proof}

The claim from the previous theorem needs to be read in the correct way: for a fixed $q$, the best choice is to take a balanced quasi-arc of parameter $(q-1)/2.$ However, the next proposition shows that when letting $q$ tending to infinity, the same limit value is obtained by taking all points on the fundamental triangle.

\begin{proposition}
Let $S$ be the set of all $3q$ points of the fundamental triangle in
$\PG(2,q)$. Then, for each fundamental point $P$,
\[
  \;\mathbb{E}[\tau_{P}]
  \;=\; \frac{17q^{2}-18q+4}{(3q-2)(2q-1)}.
  \;
\]
Hence, for $q\to \infty$, we find the same $\E[\tau_P]$ as in the case of the balanced quasi-arc.
\end{proposition}

\begin{proof}
By the symmetry of the triangle it suffices to compute $\mathbb{E}[\tau_{P_1}]$. For $s=1$ we have $\beta_{P_1}(1)=|S|-1=3q-1$.
Every line $\ell\not\ni P_1$ meets each of the three sides in a point of $S$. There is one line, $P_2P_3$ with $q+1$ points of $\mathcal{S}$, there are $2(q-1)$ lines with exactly one non-fundamental point and one fundamental point of $S$, and there are $(q-1)^2$ general lines which have $3$ points of $S$.

We find for $s=2$,
\[
  \beta_{P_1}(2)
  = \binom{q+1}{2} + 2(q-1)\binom{2}{2} + 3(q-1)^{2}
  = \frac{7q^{2}-7q+2}{2},
\]
and hence,
\[
  \frac{\beta_{P_1}(2)}{\binom{3q-1}{2}}
  = \frac{7q^{2}-7q+2}{(3q-1)(3q-2)} .
\]

For $s\ge 3$ we find
\[
  \sum_{s=3}^{N-1}\frac{\beta_{P_1}(s)}{\binom{N-1}{s}}
  = T(q+1) + (q-1)^{2}\,T(3).
\]
With $n=3q$, we have that
\[
  T(q+1)=\frac{q(q+1)(q-1)}{(3q-1)(3q-2)(2q-1)},
  \qquad
  (q-1)^{2}T(3)=\frac{2(q-1)}{(3q-1)(3q-2)},
\]
so

\[
  \mathbb{E}[\tau_{P_1}]
  = 2 + \frac{q(7q-5)}{(3q-1)(3q-2)}
      + \frac{q(q^{2}-1)}{(3q-1)(3q-2)(2q-1)} 
  = \frac{17q^{2}-18q+4}{(3q-2)(2q-1)} .
\]

\end{proof}

While the balanced quasi-arc of parameter $(q-1)/2$ is the best choice among point sets supported on the sides of the fundamental triangle, this is no longer true when we allow points not on the sides. In the next section, we will see that in general, assigning suitable weights to points that do not lie on the edges of the fundamental simplex can improve the value of $M$; see Remark~\ref{rem:k3} for the planar case.
\section{Other point sets for the DNA problem}\label{sec:constructions}
In the previous Section we saw that balanced quasi-arcs in the plane are good point sets for the random access problem, as was already noted in \cite{GruicaMontanucciZullo2026}. They also showed that by adding multiplicities to the fundamental points, the value of $\E[\tau_P]$ decreased and that the best result was obtained when taking the multiplicity equal to $a$.
 We will now consider  a $3$-dimensional analogue of this phenomenon. We first calculate the expected recovery rate for the full tetrahedron with weighted fundamental points, and then see that we have the same limit for $q$, but better expectation for fixed $q$ when we take a weak arc with $a$ points on each of the edges of the fundamental simplex, and fundamental points with weight $a.$


\subsection{The full tetrahedron with weighted fundamental points}

As in \cite{GruicaMontanucciZullo2026}, we will add multiplicities to the fundamental points. We choose to add multiplicity $q-1$ to those points because \cite{BoruchovskyEtAl2026} suggests that this will be the optimal choice for the multiplicity as it is the size of point set on each of the sides.

Hence, let $\Gscr$ be the multiset obtained from the six edges $P_iP_j$  by taking each fundamental point with multiplicity $q-1$ and each
non-fundamental point with multiplicity $1$. Counting multiplicity, $
n=|\Gscr|=4(q-1)+6(q-1)=10(q-1).$ 

\begin{proposition}\label{lem:full}
For $\Gscr\subseteq\PG(3,q)$ with every fundamental
point $P_i$ of multiplicity $q-1$, we have
\[
\E[\tau_{P_i}(\Gscr)]
=A/B,\] where \begin{align*}A&=68567620q^{7}-596753372q^{6}+2220097355q^{5}-4577254892q^{4}\\&
+5648896736q^{3}-4173395576q^{2}+1709219104q-299377920\end{align*} and \[B=
{42(3q-4)(5q-8)(5q-6)(7q-8)(9q-10)(10q-13)(10q-11)}.\]

In particular
, as $q\to \infty$, $\E[\tau_{P_i}(\Gscr)]\to \frac{3428381}{992250}\approx 3.455158$;
$\frac{\E[\tau_{P_i}(\Gscr)]}{4}\to \frac{3428381}{3969000}\approx0.863788
$.
\end{proposition}

\begin{proof}
By the symmetry of $\Gscr$, it is enough to compute $\E[\tau_{P_1}]$, where we will use the expression given in \eqref{eq:beta}. To determine $\beta_{P_1}(\Gscr,s)$
directly for each $s$ we need to find the number of lines and planes not containing $P_1$, together with the number of points of $\Gscr$ contained in them.  The $q^3$ planes not containing $P_1$ are as follows:
\begin{itemize}
    \item One plane ($P_2P_3P_4$), which has $6q-6$ points of $\Gscr$;
    \item $3(q-1)$ planes with two fundamental points, which have $3q-2$ points of $\Gscr$;
    \item $3(q-1)^2$ planes with one fundamental point, which have $q+2$ points of $\Gscr$;
    \item $(q-1)^3$ planes with no fundamental points, which have $6$ points of $\Gscr$.
\end{itemize}

The lines not containing $P_1$ but with at least two points of $\Gscr$ are as follows:
\begin{itemize}
    \item $3$ lines with $3q-3$ points of $\Gscr$ (the sides $P_2P_3,P_2P_4,P_3P_4$);
    \item $9(q-1)$ lines with $q$ points of $\Gscr$ (a fundamental point $P_2,P_3,P_4$ joined
    to a non-fundamental point of an opposite edge);
    \item $3(q-1)^2$ lines with $q-1$ points of $\Gscr$ (a fundamental point $P_2,P_3,P_4$ with
    no other point of $\Gscr$);
    \item $4(q-1)^2$ lines with $3$ points of $\Gscr$ (all lines not through a fundamental point
    contained in one of the four faces of the tetrahedron);
    \item $3(q-1)^2$ lines with $2$ points of $\Gscr$ (joining two non-fundamental points on
    opposite edges).
\end{itemize}
Finally there are three points of $\Gscr$ not containing $P_1$ of multiplicity exceeding $1$,
namely $P_2,P_3,P_4$, each with multiplicity $q-1$.

It is easy to see that $\beta_{P_1}(\Gscr,1)=9(q-1)$, so the first two terms in the expression of $\E[\tau_F]$ are given by 
\[
1+\frac{\beta_{P_1}(\Gscr,1)}{\binom{n-1}{1}}
=1+\frac{9(q-1)}{10q-11}.
\]

For $s=2$, the subsets not containing $P_1$ in their span are all pairs of the $9(q-1)$ points different from $P_1$, except those on the lines $P_1P_2,P_1P_3,P_1P_4$, all of them having $2(q-1)$ points. Of the pairs contained in those, the $\binom{q-1}{2}$ pairs made up of two copies of $P_i$ do not span $P_1$. Hence
\[
\beta_{P_1}(\Gscr,2)
=\binom{9q-9}{2}-3\!\left[\binom{2q-2}{2}-\binom{q-1}{2}\right]
=3(q-1)(12q-13),
\]
and the $s=2$ contribution is
\[
\frac{\beta_{P_1}(\Gscr,2)}{\binom{n-1}{2}}
=\frac{3(q-1)(12q-13)}{(5q-6)(10q-11)} .
\]

For $s\geq 3$, in order to count subsets of size $s$ which do not contain $P_1$ in their span, we
will distinguish between the dimension of the space the subset spans, which is now a plane, a
line, or a single point.

Since every line not containing $P_1$ lies on $q$ planes not containing $P_1$, and every point
$\ne P_1$ lies on $q^2$ planes and $q^2+q$ lines not containing $P_1$, we see that for $s\geq 3$,
\[
\beta_{P_1}(\Gscr,s)
=\sum_{\Pi\not\ni P_1}\binom{c_\Pi}{s}
-(q-1)\sum_{\ell\not\ni P_1}\binom{c_\ell}{s}
+(q-1)^2(q+1)\sum_{Q\not\ni P_1}\binom{c_Q}{s},
\]
the point coefficient being $q^3-q^2-q+1=(q-1)^2(q+1)$. Using \eqref{eq:Tdef}, we see that
\[
\sum_{s=3}^{n-1}\frac{\beta_{P_1}(\Gscr,s)}{\binom{n-1}{s}}
=\sum_{\Pi\not\ni P_1}T(c_\Pi)-(q-1)\sum_{\ell\not\ni P_1}T(c_\ell)
+3(q-1)^2(q+1)\,T(q-1).
\]

We find
\[
\begin{aligned}
\E[\tau_{P_1}]
={}&1+\frac{9(q-1)}{10q-11}+\frac{3(q-1)(12q-13)}{(5q-6)(10q-11)}+T(6q-6)+3(q-1)T(3q-2)\\&+3(q-1)^2T(q+2)+(q-1)^3T(6)\\
&{}-(q-1)\left(3\,T(3q-3)+9(q-1)T(q)+4(q-1)^2T(3)\right)\\
&+6(q-1)^2\,T(q-1).
\end{aligned}
\]
The corresponding values of $T$ are given by:

\[
\begin{array}{ll}
T(6q-6)=\dfrac{3(3q-4)(6q-7)}{2(5q-6)(10q-11)},
   & \quad
T(3q-2)=\dfrac{3(q-1)(3q-4)(3q-2)}{2(5q-6)(7q-8)(10q-11)},\\[2ex]
T(q+2)=\dfrac{q(q+1)(q+2)}{6(3q-4)(5q-6)(10q-11)},
   & \quad
T(6)=\dfrac{30}{(5q-8)(5q-6)(10q-11)},\\[2ex]
T(3q-3)=\dfrac{3(3q-5)(3q-4)}{14(5q-6)(10q-11)},
   & \quad
T(q)=\dfrac{q(q-1)(q-2)}{2(5q-6)(9q-10)(10q-11)},\\[2ex]
T(3)=\dfrac{3}{(5q-6)(10q-13)(10q-11)},
   & \quad
T(q-1)=\dfrac{(q-2)(q-3)}{18(5q-6)(10q-11)}.
\end{array}
\]

Collecting all contributions over the common denominator
$42(3q-4)(5q-8)(5q-6)(7q-8)(9q-10)(10q-13)(10q-11)$ and simplifying yields
the expression found in the statement. In partial fractions, this is
\begin{align*}
\E[\tau_{P_i}(\Gscr)]
=&\frac{3428381}{992250}
+\frac{3}{250(10q-11)}
-\frac{81}{250(10q-13)}
-\frac{10}{81(9q-10)}\\
&+\frac{30}{49(7q-8)}
-\frac{3}{125(5q-6)}
+\frac{81}{125(5q-8)}
+\frac{10}{27(3q-4)} .
\end{align*}

Letting
$q\to\infty$ leaves only the constant term $\tfrac{3428381}{992250}\approx 3.4551585$ (to which the
expectation decreases monotonically). 
\end{proof}

\subsection{The weak arc with weights on the fundamental points}

Let $P_1,\dots,P_4$ be the four fundamental points of $\PG(3,q)$, and let $\Gscr$ be a weak arc
consisting of $a\ge1$ non-fundamental points on each edge $P_iP_j$ (each of multiplicity $1$) together with
the four fundamental points, each taken with multiplicity $a$. Counting multiplicity, $
n=|\Gscr|=4a+6a=10a$.

\begin{proposition}\label{lem:weakarc}
Let $\Gscr\subseteq\PG(3,q)$ be as above then for each fundamental point $P_i$, then
\[
\E[\tau_{P_i}(\Gscr)]
=\frac{13713524a^{6}-13855940a^{5}+5676627a^{4}-1207089a^{3}+140575a^{2}-8515a+210}
{14(5a-1)(7a-1)(9a-1)(9a-2)(10a-1)(10a-3)} ,
\]
which only depends on $a$. It decreases monotonically from
$\tfrac{79}{21}\approx3.7619$ at $a=1$ to
\[
\lim_{a\to\infty}\E[\tau_{P_i}(\Gscr)]=\frac{3428381}{992250}\approx3.4551585 ;\qquad \lim_{a\to\infty}\E[\tau_{P_i}(\Gscr)]/4\approx 0.863788.
\]
\end{proposition}

\begin{proof}
By the symmetry of $\Gscr$ it is enough to compute $\E[\tau_{P_1}]$, using \eqref{eq:beta}. Since a
subset containing any copy of $P_1$ already spans $P_1$, only the $9a$ points distinct from $P_1$
contribute to $\beta_{P_1}$; these are the three vertices $P_2,P_3,P_4$ (each of multiplicity $a$)
and the $6a$ non-fundamental points.

The subspaces not containing $P_1$, spanned by subsets of $\Gscr$ are as folllows. The intersection sizes mentioned are counted with multiplicity.

\begin{itemize}
    \item Points: the three vertices $P_2,P_3,P_4$, each having intersection size $a$ with $\Gscr.$
    \item Lines: the three opposite edges $P_2P_3,P_2P_4,P_3P_4$ each with $2a+a=3a$ points of $\Gscr$; and the $9a$ lines joining a fundamental point
    $P_j$ $(j\ne1)$ to a non-fundamental point of one of the three edges not through $P_j$, each meeting $\Gscr$ in $a+1$ points. 
    \item Planes: the opposite face $P_2P_3P_4$ with $6a$ points of $\Gscr$, the $3a$ planes through one of the edges $P_2P_3,P_2P_4,P_3P_4$ together with one further non-fundamental point (each with $3a+1$ points of $\Gscr$), the $9a^2$ planes through one fundamental point and two non-fundamental points of $\Gscr$, each with $a+2$ points of $\Gscr$, and finally the $16a^3$ planes with $3$ non-fundamental points.
\end{itemize}

Using this, we see that $\beta_{P_1}(\Gscr,0)=1$, $\beta_{P_1}(\Gscr,1)=9a$. For $s=2$, all pairs of points taken from the $9a$ points different from $P_1$ will span a subspace not through $P_1$, except those contained in $P_1P_2,P_1P_3$ or $P_1P_4$. Of the $\binom{2a}{2}$ choices of pairs on $P_1P_i$, we need to exclude those made up of twice $P_i$, so we need to subtract $\binom{a}{2}$:
\[
\beta_{P_1}(\Gscr,2)
=\binom{9a}{2}
-3\left[\binom{2a}{2}-\binom{a}{2}\right]
=3a(12a-1).
\]
For $s\geq 3$, we count the subsets of size $s$ which do not contain $P_1$
in their span according to the dimension of the subspace they span. The subsets spanning a point contribute $3\binom{a}{s}$.
The subsets spanning a line are as follows. The three opposite edges
$P_2P_3,P_2P_4,P_3P_4$ each contain $3a$ points of $\Gscr$, counted with
multiplicity. On each such edge we must subtract the subsets consisting
entirely of copies of one of its two fundamental points. This gives
\[
3\left(\binom{3a}{s}-2\binom{a}{s}\right).
\]
Similarly, the $9a$ lines joining one of the fundamental points
$P_2,P_3,P_4$ to a chosen non-fundamental point on one of the three sides
not through that fundamental point contribute
\[
9a\left(\binom{a+1}{s}-\binom{a}{s}\right).
\]

We now count the subsets spanning a plane. The opposite face $P_2P_3P_4$
contains $6a$ points of $\Gscr$, counted with multiplicity. The subsets
contained in this face but not spanning the face are the subsets spanning a point of
 $P_2,P_3,P_4$, the subsets spanning one of $P_2P_3,P_3P_4,P_2P_4$,
and the subsets spanning one of the $3a$ lines joining one of $P_2,P_3,P_4$ to a chosen
non-fundamental point on the opposite side of the face $P_2P_3P_4$. Hence, $P_2P_3P_4$
contributes
\[
\binom{6a}{s}
-3\binom{a}{s}
-3\left(\binom{3a}{s}-2\binom{a}{s}\right)
-3a\left(\binom{a+1}{s}-\binom{a}{s}\right).
\]

There are $3a$ planes through one of the edges
$P_2P_3,P_2P_4,P_3P_4$ together with one chosen non-fundamental point on the
opposite side through $P_1$. Each such plane contains $3a+1$ points of
$\Gscr$. In such a plane, the proper subspaces spanned by subsets of those $3a+1$ points are the subsets contained in the edge $P_iP_j$ itself, the two lines joining the chosen non-fundamental point to one of the two fundamental points $P_i,P_j$, and the subsets composed of $s$ copies of either $P_i$ or $P_j$. Thus these planes contribute
\[
3a\left(
\binom{3a+1}{s}
-\left(\binom{3a}{s}-2\binom{a}{s}\right)
-2\left(\binom{a+1}{s}-\binom{a}{s}\right)-2\binom{a}{s}
\right).
\]

There are $9a^2$ planes through one fundamental point and two chosen
non-fundamental points of $\Gscr$. Each contains $a+2$ points of $\Gscr$.
These
planes contribute
\[
9a^2\left(
\binom{a+2}{s}
-2\left(\binom{a+1}{s}-\binom{a}{s}\right)-\binom{a}{s}
\right).
\]

Finally, the $16a^3$ planes with three non-fundamental points contribute
\[
16a^3\binom{3}{s}.
\]

Therefore, for $s\geq 3$,
\[
\begin{aligned}
\beta_{P_1}(\Gscr,s)
={}&
3\binom{a}{s}
+3\left(\binom{3a}{s}-2\binom{a}{s}\right)
+9a\left(\binom{a+1}{s}-\binom{a}{s}\right)\\
&{}+\binom{6a}{s}
-3\binom{a}{s}
-3\left(\binom{3a}{s}-2\binom{a}{s}\right)
-3a\left(\binom{a+1}{s}-\binom{a}{s}\right)\\
&{}+3a\left(
\binom{3a+1}{s}
-\binom{3a}{s}
-2\binom{a+1}{s}
+2\binom{a}{s}
\right)\\
&{}+9a^2\left(
\binom{a+2}{s}
-2\binom{a+1}{s}
+\binom{a}{s}
\right)
+16a^3\binom{3}{s}.
\end{aligned}
\]

\[
\begin{aligned}
\beta_{P_1}(\Gscr,s)
={}&
\binom{6a}{s}
+3a\binom{3a+1}{s}
+9a^2\binom{a+2}{s}
+16a^3\binom{3}{s}\\
&{}-3a\binom{3a}{s}
-18a^2\binom{a+1}{s}
+9a^2\binom{a}{s}.
\end{aligned}
\]
We find, using \eqref{eq:Tdef}

\[
\begin{aligned}
\E[\tau_{P_1}(\Gscr)]
={}&1+\frac{9a}{10a-1}
+\frac{3a(12a-1)}{(10a-1)(5a-1)}\\
&{}+T(6a)
+3aT(3a+1)
+9a^2T(a+2)
+16a^3T(3)\\
&{}-3aT(3a)
-18a^2T(a+1)
+9a^2T(a).
\end{aligned}
\]

We find the following values, with $n=10a$,


\[
\begin{array}{ll}
T(6a)=\dfrac{3(6a-1)(3a-1)}{2(10a-1)(5a-1)},
   & \quad
T(3a+1)=\dfrac{3a(3a-1)(3a+1)}{2(10a-1)(5a-1)(7a-1)},\\[2ex]
T(a+2)=\dfrac{a(a+1)(a+2)}{2(10a-1)(5a-1)(9a-2)},
   & \quad
T(3)=\dfrac{3}{(10a-1)(5a-1)(10a-3)},\\[2ex]
T(3a)=\dfrac{3(3a-1)(3a-2)}{14(10a-1)(5a-1)},
   & \quad
T(a+1)=\dfrac{a(a-1)(a+1)}{2(10a-1)(5a-1)(9a-1)},\\[2ex]
T(a)=\dfrac{(a-1)(a-2)}{18(10a-1)(5a-1)}.
   &
\end{array}
\]
Hence, 
\[
\begin{aligned}
\E[\tau_{P_1}(\Gscr)]
={}&1+\frac{9a}{10a-1}
+\frac{3a(12a-1)}{(10a-1)(5a-1)}
+\frac{3(6a-1)(3a-1)}
{2(10a-1)(5a-1)}+\frac{9a^2(3a-1)(3a+1)}
{2(10a-1)(5a-1)(7a-1)}\\
&{}+\frac{9a^3(a+1)(a+2)}
{2(10a-1)(5a-1)(9a-2)}+\frac{48a^3}
{(10a-1)(5a-1)(10a-3)}-\frac{9a(3a-1)(3a-2)}
{14(10a-1)(5a-1)}\\
&{}-\frac{9a^3(a-1)(a+1)}
{(10a-1)(5a-1)(9a-1)}+\frac{a^2(a-1)(a-2)}
{2(10a-1)(5a-1)}.
\end{aligned}
\]

In partial fractions,
\[
\begin{aligned}
\E[\tau_{P_1}(\Gscr)]
={}&
\frac{3428381}{992250}
+\frac{6}{125(10a-1)}
+\frac{162}{125(10a-3)}\\
&{}-\frac{20}{81(9a-1)}
+\frac{80}{81(9a-2)}
+\frac{30}{49(7a-1)}
-\frac{48}{125(5a-1)}.
\end{aligned}
\]

Letting $a\to\infty$ gives
$\tfrac{3428381}{992250}\approx3.4551585$, to which the expectation decreases .
\end{proof}

  \begin{remark}
      The asymptotic values found in the two previous constructions (for $q\to \infty$ and for $a\to \infty $ respectively) are exactly the same. However, it is true that for given size of the chosen point-set, the weak arc construction will give a smaller value of $\E[\tau_{P_1}(\Gscr)]$; this can be seen by evaluating the polynomial found in Proposition \ref{lem:weakarc} for $a=q$ and comparing it to the polynomial found in Proposition \ref{lem:full}.
  \end{remark}

  \begin{remark} 
This construction is closely related to the recovery-complete matrices
 $G_4(a,a)$ introduced in \cite{BoruchovskyEtAl2026}.  In their construction, the matrix $G_4(x,y)$ is built on $K_4$, with $x$ columns on each of the six edges
and $y$ copies of each vertex, so its length is $6x+4y$. Identifying the
six edges of $K_4$ with the edges of the fundamental tetrahedron and its
four vertices with the fundamental points $P_1,\dots,P_4$, our
construction places $a$ non-fundamental points on each edge and gives
each fundamental point multiplicity $a$; it thus has length $6a+4a=10a$ and realises $G_4(a,a)$.

The two objects arise differently. Being \emph{recovery complete}, the
matrices of \cite{BoruchovskyEtAl2026} admit a purely combinatorial recovery rule: a requested vertex $P_i$ is recoverable from a sampled
submultiset precisely when the connected component of $P_i$ in the
induced submultigraph of $K_4$ contains a cycle. Our construction obeys the same rule but for geometric reasons, via the weak-cap and weak-arc properties; their recovery condition (the component of $P_i$ contains a cycle) is equivalent to the geometric recovery condition ($P_i$ lies in the span of the sample).
The recovery expectation we have calculated above was computed directly from the incidence structure of points, lines and planes, without needing the algebraic/graph-theoretic  machinery of
\cite{BoruchovskyEtAl2026}, we also do not need a large field to construct our set. Of course, as was to be expected, our limiting value for recovery expectation coincide with that found for $G_4(a,a)$ so can be seen as an independent verification of their result.

  \end{remark}

\subsection{A geometric view on the Wang-Yaakobi construction}\label{subs:WY}
In the previous constructions, our point set was sparse in $\PG(3,q)$. Recently, Wang and Yaakobi gave a construction with a particularly good  recovery expectation where they considered the {\em entire} point set of $\PG(3,q)$ with multiplicities.
We now recall the construction of Wang and Yaakobi \cite{WangYaakobi2026} and reinterpret it geometrically in terms of the fundamental tetrahedron in $\PG(3,q)$. Let $w_1,\ldots,w_k$ be nonnegative rational numbers
satisfying
\[
 \sum_{j=1}^{k}\binom{k}{j}(q-1)^j w_j=1.
\]
Choose a positive integer $n$ such that $nw_j$ is an integer for every $j$.
The associated generator matrix $G_{\mathbf w}$ is obtained by taking, for
every nonzero vector $\mathbf u\in\F_q^k$ of Hamming weight $j$, exactly
$nw_j$ copies of $\mathbf u$ as columns. 

In geometric terms, the vectors with a given Hamming weight $j$ correspond respectively, to the vertices of the fundamental tetrahedron ($j=1$), non-fundamental points on an edge ($j=2$), points on a face not an edge ($j=3$) and finally, the points not on any face of the fundamental tetrahedron ($j=4$). 

Wang and Yaakobi derive a general upper-bound formula for such distributions
and optimize the parameters $w_j$. For $k=4$, however, their paper records the
result only in a remark, stating that the method yields
\[
 \limsup_{q\to\infty}M(\Gscr(q,4))<0.862882\cdot4.
\]
They do not provide an explicit optimizing quadruple
$(w_1,w_2,w_3,w_4)$.

Using the same methods as used for the point sets in the previous Subsections, we can show that the following discrete version of their construction, with integer multiplicities, comes very close to their optimised distribution.

Let $\Omega_j$ be the set of points in $\PG(3,q)$ determined by a vector of Hamming weight $j$. Then
\[
|\Omega_1|=4,\qquad |\Omega_2|=6(q-1),\qquad
|\Omega_3|=4(q-1)^2,\qquad |\Omega_4|=(q-1)^3.
\]

Define the projective multiset
\begin{equation}
\label{eq:full-strata-construction}
\mathcal H_q
 =5(q-1)^3\Omega_1\cup4(q-1)^2\Omega_2\cup3\Omega_4.
\end{equation}
Note that we do not take points on the face of the tetrahedron, except those on edges.
We see that the size of the set is $47(q-1)^3$.
 
\begin{proposition}\label{prop:fullstrata}
Let $\mathcal H_q\subseteq\PG(3,q)$ be the multiset \eqref{eq:full-strata-construction}.
Then for each fundamental point $P_i$ the value $\E[\tau_{P_i}(\mathcal H_q)]$ is
independent of $i$ and we have
\[
\lim_{q\to\infty}\E[\tau_{P_i}(\mathcal H_q)]
   =\frac{37968534049}{11001615240}\approx 3.4511781 ,
\quad
\lim_{q\to\infty}\frac{\E[\tau_{P_i}(\mathcal H_q)]}{4}
   \approx 0.8627945 .
\]
\end{proposition}

\begin{proof} The proof of this Proposition goes in the same way as the cases covered in the previous subsections. We collected the details in the Appendix.
\end{proof}
\begin{remark}\label{rem:k3}
We can also consider the planar analogue of the construction given above: we put weight $w_1$ on each fundamental point, $w_2$ on each non-fundamental
point on a side of the fundamental triangle, and $w_3$ on each of the $(q-1)^2$ remaining points 
with $n=3w_1+3(q-1)w_2+(q-1)^2 w_3$ the total weight of the set we obtain. Optimising the limit yields
\[
  w_1:w_2:w_3 \;\approx\; 2.3943\,q^2 \;:\; 2.2540\,q \;:\; 1,
  \qquad
  \lim_{q\to\infty}\tfrac{1}{3}\E[\tau_{P_1}]\;\approx\;0.881090,
\]
which is (much) better than the value $17/18$ found in the unweighted balanced quasi-arc construction in Subsection \ref{subs:planar}. However, almost all of the gain
already comes from weighting the vertices and edges: if we are choosing integer weights, then putting $w_1=19q^2$, $w_2=18q$ and $w_3=0$ results in a limit of $0.88155$, which is slightly better than the $0.8822$ as found in \cite{GruicaMontanucciZullo2026}.
\end{remark}

{

\section{Comparison with previous constructions}
\label{sec:comparison}
We briefly compare our constructions with some of the previously known ones for the random-access problem. The comparison does not just consider the recovery expectation, but also the length of the construction, the required field size, and the availability of explicit examples.

The required field size is an important parameter; see
also \cite{BertuzzoRavagnaniYaakobi2025} for results specifically
addressing the small-alphabet regime. 

The weighted weak arcs of Proposition~\ref{lem:weakarc} provide a
finite-geometric realisation of the recovery-complete family $G_4(a,a)$
of \cite{BoruchovskyEtAl2026}. Their asymptotic normalised expectation is
\[
\frac{\E[\tau_{P_i}]}4\longrightarrow 0.863788,
\]
which is very close to the value obtained by optimising the more general
family $G_4(x,y)$, approximately $0.86375$. The main advantage of our
construction lies instead in the field size: for instance, $a=4$ is
realised explicitly over $\F_{17}$, whereas the general sufficient
construction in \cite{BoruchovskyEtAl2026} gives the much larger field
size $q=2^{20}$. This latter bound is only sufficient and is not claimed
to be minimal.

Our second construction, $\mathcal H_q$, fits in the framework of Wang and Yaakobi \cite{WangYaakobi2026}. But it is explicit, has integer multiplicities, and is defined over every finite field. Moreover,
\[
\lim_{q\to\infty}\frac{\E[\tau_{P_i}(\mathcal H_q)]}{4}
\approx0.862795,
\]
which slightly improves the explicit numerical bound $0.862882$ stated in \cite{WangYaakobi2026}. The price to pay is its considerably larger length,
\[
|\mathcal H_q|=47(q-1)^3.
\]
We see that the weak-arc construction is preferable when small support and small fields are important, while $\mathcal H_q$ gives the better asymptotic recovery expectation.

\begin{table}[H]
\centering
\caption{Comparison of constructions for four information strands.}
\label{tab:k4-comparison}
\begin{tabular}{@{}lccc@{}}
\toprule
Construction & Length & Field & $\displaystyle\lim \E[\tau_{P_i}]/4$\\
\midrule
Full tetrahedron (Prop.~\ref{lem:full})
& $10(q-1)$ & any $q$ & $0.863788$\\
Weighted weak arc (Prop.~\ref{lem:weakarc})
& $10a$ & explicit for small $q$ & $0.863788$\\
$G_4(x,y)$ \cite{BoruchovskyEtAl2026}
& $6x+4y$ & general large-field construction & $\approx0.86375$\\
$\mathcal H_q$ (Prop.~\ref{prop:fullstrata})
& $47(q-1)^3$ & any $q$ & ${0.862795}$\\
Wang--Yaakobi \cite{WangYaakobi2026}
& --- & large $q$ & $<0.862882$\\
\bottomrule
\end{tabular}
\end{table}

}

\newpage
\appendix
\section{Lemma used in the proof of Proposition \ref{prop:balbetter}}
\begin{lemma} \label{lem:phi}Let $\phi(t):=T(t+2)-\frac{2\binom{t+1}{2}}{\binom{n-1}{2}}.$
Then $\sum_{r}\phi(a_r)\ge3\phi(a)$ when $a_1+a_2+a_3=3a$, and the inequality is strict unless all $a_1=a_2=a_3=a$.
\end{lemma}
\begin{proof}

Using that $T(t+2)=\frac{t(t+1)(t+2)}{(n-1)(n-2)(n-t-2)}$ and $\frac{2\binom{t+1}{2}}{\binom{n-1}{2}}
= \frac{2t(t+1)}{(n-1)(n-2)},$
we find
\[
\phi(t)
=\frac{t(t+1)}{(n-1)(n-2)}
\left(\frac{t+2}{n-t-2}-2\right).
\]
Since $n=3a+3$, this becomes $
\phi(t)
=\frac{3t(t+1)(t-2a)}{(3a+2)(3a+1)(3a+1-t)}.$
The constant factor $
\frac{3}{(3a+2)(3a+1)}$
is positive, so it is enough to show that
$
\sum_{r=1}^3 g(a_r)\ge 3g(a)$ for 
\[
g(t)=\frac{t(t+1)(t-2a)}{3a+1-t}.
\]
W.l.o.g. we may assume that
$a_1\le a_2\le a_3.
$ Since $a_1+a_2+a_3=3a$, we may write
$
a_1=a-u, a_2=a-v, a_3=a+u+v
$
for some $u,v$ where $0\leq u\leq a$ and $v\leq u$. A direct calculation gives

\begin{align}
&g(a-u)+g(a-v)+g(a+u+v)-3g(a)\nonumber \\
&\qquad =
\frac{P_a(u,v)}
{(2a+1)(2a+u+1)(2a+v+1)(2a+1-u-v)}, \label{app}
\end{align}

where
\[
\begin{aligned}
P_a(u,v)={}&4a^4(u^2+uv+v^2) +a^3\left(27(u^2v+uv^2)+26(u^2+uv+v^2)\right)\\
&+a^2\left(8(u^4+v^4)+16(u^3v+uv^3)+24u^2v^2+54(u^2v+uv^2)+32(u^2+uv+v^2)\right)\\
&+a\left(4(u^4v+uv^4)+8(u^4+v^4)+8(u^3v^2+u^2v^3) +16(u^3v+uv^3)+24u^2v^2\right.\\
&\left.\qquad +33(u^2v+uv^2)+14(u^2+uv+v^2)\right)\\
&+2(u^4v+uv^4)+2(u^4+v^4)+4(u^3v^2+u^2v^3)\\
& +4(u^3v+uv^3)+6u^2v^2+6(u^2v+uv^2)+2(u^2+uv+v^2).
\end{aligned}
\]

Using $X=u^2+uv+v^2$, and $Y=uv(u+v)$, we can rewrite $P_a(u,v)$ as
\[
P_a(u,v)=C(a)X+B(a)Y+2(2a+1)^2X^2+2(2a+1)XY,
\]
where
\[
C(a)=2(2a+1)(a^3+6a^2+5a+1)
\]
and
\[
B(a)=3(a+1)(3a+1)(3a+2).
\]
Since $a\geq0$, both $C(a)$ and $B(a)$ are strictly positive.

Suppose first that in addition to $u\geq 0$, also $v\geq0$. Then $X,Y\geq 0$ and it follows immediately that $P_a(u,v)\geq 0$. Moreover, if equality holds, then $C(a)X=0$, so $X=0$, which in turn forces $u=v=0.$

It remains to consider $-u\leq v<0$.  Write $w=-v$.  We know that $0<w\leq u\leq a$ and $
X=u^2-uw+w^2$. Define $\tilde{Y}=uw(u-w)=-Y$. We have
\[
P_a(u,v)
=C(a)X-B(a)\tilde{Y}+2(2a+1)^2X^2-2(2a+1)X\tilde{Y}.
\]
Moreover,
\[
uX-3\tilde{Y}
=u(u^2-uw+w^2)-3uw(u-w)
=u(u-2w)^2\geq0,
\]
so
\[
\tilde{Y}\leq \frac{u}{3}X.
\]
Also,
\[
X=\left(w-\frac{u}{2}\right)^2+\frac{3u^2}{4}
\geq\frac{3u^2}{4}.
\]
Using $\tilde{Y}\leq uX/3$, we obtain
\[
\begin{aligned}
P_a(u,v)
&\geq X\left[
C(a)-\frac{B(a)u}{3}
+2(2a+1)\left(2a+1-\frac{u}{3}\right)X
\right].
\end{aligned}
\]
Because $0\leq u\leq a$ and $a\geq0$,
\[
2a+1-\frac{u}{3}>0.
\]
We may therefore use $X\geq3u^2/4$ to deduce
\[
P_a(u,v)\geq X\left[
C(a)-\frac{B(a)u}{3}
+\frac{3}{2}(2a+1)u^2
 \left(2a+1-\frac{u}{3}\right)
\right].
\]
The expression in square brackets can be rewritten as
\[
\begin{aligned}
&C(a)-\frac{B(a)u}{3}
+\frac{3}{2}(2a+1)u^2
 \left(2a+1-\frac{u}{3}\right)\\
&\qquad=
\frac{4a-u+2}{2}
\left[
9a^2+9a+2
+(2a+1)\bigl((a-u)^2+(a-u)\bigr)
\right].
\end{aligned}
\]
Every factor on the right-hand side is strictly positive, as is $X$, hence, $P_a(u,v)>0$ in this case. We conclude that that $P_a(u,v)\geq0$ with equality if and only if $(u,v)=(0,0)$.

Since $u=v=0$ if and only if $a_1=a_2=a_3=a$, and the denominator in \eqref{app} is positive, we conclude that

\[
\sum_{r=1}^3 g(a_r)- 3g(a)\geq 0.
\]

with equality if and only if $a_1=a_2=a_3=a.$

\end{proof}

\section{Proof of Proposition \ref{prop:fullstrata}}
\begin{proof}
We follow the same strategy as in the proof of Proposition \ref{lem:full}: we consider the different types of points, lines and planes not containing $P_1$, and the subsets of size $s$ contained within each that span that subspace.

We list the points, lines and planes not containing $P_1$, together with their
$\mathcal H_q$-intersection size $c$, counted with multiplicity.

\emph{Points $\ne P_1$.}
\[
\begin{array}{lcl}
\text{type} & \#\ & c\\\hline
\text{vertices }P_2,P_3,P_4 & 3           & w_1=5(q-1)^3\\
\text{edge points}          & 6(q-1)      & w_2=4(q-1)^2\\
\text{weight-$3$ points}    & 4(q-1)^2    & 0\\
\text{weight-$4$ points}    & (q-1)^3     & w_4=3
\end{array}
\]

\emph{Lines $\ell\not\ni P_1$.} Note that each line in $\PG(3,q)$ meets each face of the tetrahedron, and that those have been taken with weight $0$.
\[
\begin{array}{lcl}
\text{line type} & \#\text{ lines} & c_\ell\\\hline
\text{opposite edges }P_2P_3,P_2P_4,P_3P_4 & 3            & 2w_1+(q-1)w_2\\
\text{through a fundamental point, in a face} & 9(q-1)      & w_1+w_2\\
\text{through a fundamental point, not in face}            & 3(q-1)^2    & w_1+(q-1)w_4\\
\text{meeting two opposite edges}           & 3(q-1)^2    & 2w_2+(q-1)w_4\\
\text{in face, generic}                     & 4(q-1)^2    & 3w_2\\
\text{through edge point, generic}              & 6(q-1)^3    & w_2+(q-2)w_4\\
\text{generic}                              & (q-1)^3(q-2)& (q-3)w_4
\end{array}
\]

\emph{Planes $\Pi\not\ni P_1$}
\[
\begin{array}{lcl}
\text{plane type} & \#\text{ planes} & c_\Pi\\\hline
\text{opposite face }P_2P_3P_4 & 1          & 3w_1+3(q-1)w_2\\
\text{containing one edge}        & 3(q-1)     & 2w_1+q\,w_2+(q-1)^2w_4\\
\text{through one fundamental point}      & 3(q-1)^2   & w_1+3w_2+(q-1)(q-2)w_4\\
\text{generic}       & (q-1)^3    & 6w_2+(q^2-3q+3)w_4
\end{array}
\]

Clearly $\beta_{P_1}(0)=1$ and $\beta_{P_1}(1)=n-w_1=42(q-1)^3$.
For $s=2$ every pair of the $42(q-1)^3$ points $\ne P_1$ spans a subspace not containing
$P_1$, except for those pairs of distinct points spanning a line through $P_1$. There are two types of such lines: $3$ contain $P_1$ and a second fundamental point, and $(q-1)^2$ contain $P_1$ and are not contained in a face of the tetrahedron. We get
\begin{align*}
\beta_{P_1}(2)
&= \\&=\frac{3(q-1)^3\bigl(532q^3-1580q^2+1561q-524\bigr)}{2}.
\end{align*}

For $s\ge3$ we use the same
inclusion--exclusion as in the full-tetrahedron and find
\[
\sum_{s=3}^{n-1}\frac{\beta_{P_1}(\mathcal H_q,s)}{\binom{n-1}{s}}
 =\sum_{\Pi\not\ni P_1}T(c_\Pi)
 -(q-1)\sum_{\ell\not\ni P_1}T(c_\ell)
 +(q-1)^2(q+1)\sum_{Q\ne P_1}T(c_Q).
\]
Substituting all values from above we see,
\[
\begin{aligned}
\E[\tau_{P_1}]
={}&1+\frac{42(q-1)^3}{n-1}+\frac{\beta_{P_1}(2)}{\binom{n-1}{2}}\\
&+T\!\big(3w_1+3(q-1)w_2\big)
 +3(q-1)\,T\!\big(2w_1+q\,w_2+(q-1)^2w_4\big)\\
&+3(q-1)^2\,T\!\big(w_1+3w_2+(q-1)(q-2)w_4\big)
 +(q-1)^3\,T\!\big(6w_2+(q^2-3q+3)w_4\big)\\
&-(q-1)\Big[\,3\,T\!\big(2w_1+(q-1)w_2\big)+9(q-1)\,T(w_1+w_2)
      +3(q-1)^2\,T\!\big(w_1+(q-1)w_4\big)\\
&\qquad\quad +3(q-1)^2\,T\!\big(2w_2+(q-1)w_4\big)+4(q-1)^2\,T(3w_2)\\
&\qquad\quad +6(q-1)^3\,T\!\big(w_2+(q-2)w_4\big)
      +(q-1)^3(q-2)\,T\!\big((q-3)w_4\big)\Big]\\
&+(q-1)^2(q+1)\Big[\,3\,T(w_1)+6(q-1)\,T(w_2)+(q-1)^3\,T(w_4)\,\Big].
\end{aligned}
\]

Writing $D:=(n-1)(n-2)=(47q^3-141q^2+141q-48)(47q^3-141q^2+141q-49)$, we find
\[
\begin{array}{l}
T(w_1)=\dfrac{5(5q^3-15q^2+15q-7)(5q^3-15q^2+15q-6)}{42\,D},\\
T(w_2)=\dfrac{8(2q-3)(2q-1)(2q^2-4q+1)}{(47q-51)\,D},\qquad
T(w_4)=T(3)=\dfrac{6}{(47q^3-141q^2+141q-50)\,D},\\
T\!\big(3w_1+3(q-1)w_2\big)=\dfrac{27(3q-4)(9q^2-15q+7)(27q^3-81q^2+81q-29)}{20\,D},\\
T\!\big(2w_1+q\,w_2+(q-1)^2w_4\big)
   =\dfrac{7(2q-1)(14q^3-35q^2+28q-9)(14q^3-35q^2+28q-8)}{(33q-40)\,D},\\
T\!\big(w_1+3w_2+(q-1)(q-2)w_4\big)
   =\dfrac{(5q^2+5q-13)(5q^3-18q+11)(5q^3-18q+12)}{3(14q^2-33q+20)\,D},\\
T\!\big(6w_2+(q^2-3q+3)w_4\big)
   =\dfrac{3(9q^2-19q+11)(27q^2-57q+31)(27q^2-57q+32)}{(47q^3-168q^2+198q-80)\,D},\\
T\!\big(2w_1+(q-1)w_2\big)
   =\dfrac{28(7q^3-21q^2+21q-8)(14q^3-42q^2+42q-15)}{33\,D},\\
T(w_1+w_2)=\dfrac{(5q-1)(5q^3-11q^2+7q-3)(5q^3-11q^2+7q-2)}{2(21q-23)\,D},\\
T\!\big(w_1+(q-1)w_4\big)
   =\dfrac{(5q^2-10q+8)(5q^3-15q^2+18q-10)(5q^3-15q^2+18q-9)}{3(14q^2-28q+13)\,D},\\
T\!\big(2w_2+(q-1)w_4\big)
   =\dfrac{(8q-5)(8q^2-13q+3)(8q^2-13q+4)}{(47q^2-102q+52)\,D},\\
T(3w_2)=\dfrac{24(6q^2-12q+5)(12q^2-24q+11)}{(47q-59)\,D},\\
T\!\big(w_2+(q-2)w_4\big)
   =\dfrac{(4q^2-5q-4)(4q^2-5q-3)(4q^2-5q-2)}{(47q^3-145q^2+146q-45)\,D},\\
T\!\big((q-3)w_4\big)=\dfrac{3(q-3)(3q-11)(3q-10)}{(47q^3-141q^2+138q-38)\,D}.
\end{array}
\]

We can collect these terms and find the partial-fraction decomposition

\begin{align*}
&\E[\tau_{P_i}(\mathcal H_q)]\\
&=\frac{37968534049}{11001615240} + \frac{-3411 q^2 + 8703 q - 5562}{2209 (47 q^3 - 141 q^2 + 138 q - 38)} \\&+ \frac{11467308 q^2 - 27062628 q + 15125322}{4879681 (47 q^3 - 145 q^2 + 146 q - 45)} + \frac{2303}{363 (33 q - 40)} \\&- \frac{188}{1029 (21 q - 23)} + \frac{846 q^2 - 1692 q + 873}{2209 (47 q^3 - 141 q^2 + 141 q - 50)} \\&+ \frac{150528}{4879681 (47 q - 51)} + \frac{188 - 705 q}{2744 (14 q^2 - 33 q + 20)} \\&- \frac{82944}{103823 (47 q - 59)} + \frac{-47 q + 47}{196 (14 q^2 - 28 q + 13)}\\&+ \frac{128235 - 153147 q}{103823 (47 q^2 - 102 q + 52)} + \frac{600813 q^2 - 1174806 q + 630081}{103823 (47 q^3 - 168 q^2 + 198 q - 80)}
\end{align*}

The leading constant is exactly the limit $\tfrac{37968534049}{11001615240}$, and
every remaining term tends to $0$ as $q\to\infty$.
\end{proof}
We can see that the first values of the expectation and of $\E[\tau_{P_i}]/4$ are
\[
\renewcommand{\arraystretch}{1.3}
\begin{array}{lcccc}
q & 2 & 3 & 4 & 7\\\hline
\E[\tau_{P_i}] & 3.9075439 & 3.6047582 & 3.5455837 & 3.4952331\\
\E[\tau_{P_i}]/4 & 0.9768860 & 0.9011896 & 0.8863959 & 0.8738083
\end{array}
\]
decreasing to $\tfrac{37968534049}{44006460960}\approx0.8627945$.

\begin{remark} The calculations in this paper are elementary but tedious to do by hand. All the algebraic manipulations in this paper have been executed with Maple (2025).
\end{remark}
\end{document}